\documentclass[a4paper,11pt,reqno]{amsart}
\usepackage{amsrefs}
\usepackage{mathrsfs}
\usepackage{amsfonts}
\usepackage[centertags]{amsmath}
\usepackage{amssymb}
\usepackage{amsthm}
\usepackage{graphicx}
\usepackage{caption}
\usepackage{hyperref}
\usepackage{enumerate}
\usepackage{geometry}
\usepackage{appendix}
\usepackage{color}
\usepackage[all]{xy}
\usepackage{float}
\usepackage{amsmath}
\usepackage{longtable}
\usepackage{resizegather}
\usepackage{framed}
\usepackage{mathdots}
\usepackage{multirow}
\usepackage{array}
\usepackage{mathtools}
\usepackage{tikz}
\usepackage{dashrule}
\usepackage{young}
\usetikzlibrary{chains}

\newtheorem{theorem}{Theorem}[section]
\newtheorem{corollary}[theorem]{Corollary}
\newtheorem{lemma}[theorem]{Lemma}
\newtheorem{proposition}[theorem]{Proposition}
\newtheorem{definition}[theorem]{Definition}
\newtheorem{remark}[theorem]{Remark}
\newtheorem{example}[theorem]{Example}

\numberwithin{equation}{section}

\newcommand\scalemath[2]{\scalebox{#1}{\mbox{\ensuremath{\displaystyle #2}}}}
\newcommand{\udots}{\mathinner{\mskip1mu\raise1pt\vbox{\kern7pt\hbox{.}}
\mskip2mu\raise4pt\hbox{.}\mskip2mu\raise7pt\hbox{.}\mskip1mu}}

\title{Quiver mutations and Boolean reflection monoids}
\author{Bing Duan, Jian-Rong Li, and Yan-Feng Luo}
\address{Bing Duan: School of Mathematics and Statistics, Lanzhou University, Lanzhou 730000, P. R. China.}
\email{duan890818@163.com}
\address{Jian-Rong Li: Dept. of Mathematics, The Weizmann Institute of Science, Rehovot 7610001, Israel; school of Mathematics and Statistics, Lanzhou University, Lanzhou 730000, P. R. China.}
\email{lijr07@gmail.com}
\address{Yan-Feng Luo: School of Mathematics and Statistics, Lanzhou University, Lanzhou 730000, P. R. China.}
\email{luoyf@lzu.edu.cn}
\date{}

\begin{document}

\maketitle

\begin{abstract}
In 2010, Everitt and Fountain introduced the concept of reflection monoids. The Boolean reflection monoids form a family of reflection monoids (symmetric inverse semigroups are Boolean reflection monoids of type $A$). In this paper, we give a family of presentations of Boolean reflection monoids and show how these presentations are compatible with quiver mutations of orientations of Dynkin diagrams with frozen vertices. Our results recover the presentations of Boolean reflection monoids given by Everitt and Fountain and the presentations of symmetric inverse semigroups given by Popova respectively. Surprisingly, inner by diagram automorphisms of irreducible Weyl groups and Boolean reflection monoids can be constructed by sequences of mutations preserving the same underlying diagrams. Besides, we show that semigroup algebras of Boolean reflection monoids are cellular algebras.

\hspace{0.15cm}


\hspace{0.15cm}

\end{abstract}

\section{Introduction}
Around 2000, Fomin and Zelevinsky introduced a new family of commutative algebras called cluster algebras, \cite{FZ02}.  The theory of cluster algebras has connections and applications to diverse areas of mathematics and physics, including quiver representations, Teichm\"{u}ller theory, tropical geometry, integrable systems, and Poisson geometry. 

Barot and Marsh applied the idea in cluster algebras to give new presentations of finite irreducible crystallographic reflection groups, \cite{BM15}. 

Quiver mutations play an important role in the theory of cluster algebras and the work of Barot and Marsh \cite{BM15}. Finite type cluster algebras are classified by finite type Dynkin diagrams, \cite{FZ03}. The quivers appearing in the work of Barot and Marsh \cite{BM15} are
quivers of finite type (i.e., these quivers are mutation equivalent to orientations of finite
type Dynkin diagrams).

Let $\Gamma$ be a quiver whose underlying graph is a Dynkin diagram. In \cite{BM15}, Barot and Marsh gave new presentations of the reflection group $W_\Gamma$ determined by $\Gamma$ and showed that these presentations are compatible with mutations of $\Gamma$  quivers. More precisely, Barot and Marsh introduced some additional relations (cycle relations) corresponding to chordless cycles arising in quivers of finite type. For any quiver $Q$ of finite type, they defined an abstract group $W(Q)$ by generators (corresponding to vertices of $Q$) and relations and then proved that $W(Q) \cong W_\Gamma$.

Motivated by Barot and Marsh's work, new presentations of affine Coxeter groups, braid groups, Artin groups, and Weyl groups of Kac-Moody algebras have been studied using quiver mutations in \cite{FT13,FT14,GM14,HHLP,S14}, respectively.

%It is well known that finite irreducible crystallographic reflection groups or irreducible Weyl groups are classified by Dynkin diagrams, whose vertex set is in one-to-one correspondence with the set $S$ of simple reflections and for which there is an edge labeled $1$ (respectively, $2$, $3$) between vertices $i$ and $j$ if and only if $(s_is_j)^3=e$ (respectively, $(s_is_j)^4=e$, $(s_is_j)^6=e$) where $s_i, s_j \in S$ and $e$ is the identity element, see \cite{BM15,C35,D08,H90}.

One of the aims in this paper is to give new presentations of Boolean reflection monoids and show that these presentations are compatible with mutation of certain quivers. 

In \cite{EF10}, Everitt and Fountain introduced the concept of reflection monoids. The Boolean reflection monoids are reflection monoids. Symmetric inverse semigroups are Boolean reflection monoids of type $A$. In \cite{EF13}, Everitt and Fountain obtained presentations of the Boolean reflection monoids of type $A, B/C, D$.

Let $A^\varepsilon_{n-1}$ (respectively, $B^\varepsilon_n$, $D^\varepsilon_n$) be the Dynkin diagram with $n$ (respectively, $n+1$, $n+1$) vertices, among them the first $n-1$ (respectively, $n$, $n$) vertices are mutable vertices and the last vertex $\varepsilon$ is a frozen vertex, which is shown in the 4-th column of Table \ref{initial ABD}. We label an edge if its weight is greater or equal to 2.

Let $\Delta \in \{A^\varepsilon_{n-1},B^\varepsilon_n,D^\varepsilon_n\}$ and $Q$ any quiver that is mutation equivalent to a quiver whose underlying graph is $\Delta$. We define an inverse monoid $M(Q)$ from $Q$, see Section \ref{monoids and diagrams}, and show that $M(Q) \cong M(\Phi,\mathcal{B})$, where $M(\Phi,\mathcal{B})$ is the Boolean reflection monoid defined in \cite{EF10}, see Theorem \ref{mutation-invariance} and Theorem \ref{Everitt and Fountain's presentations}. This implies that Boolean reflection monoids can also be classified by $\Delta$, see Table \ref{initial ABD}. In \cite{BM15, FT13, GM14, HHLP}, the diagrams corresponding to presentations of irreducible Weyl groups, affine Coxeter groups, braid groups, Artin groups have no frozen vertices. In the present paper, the diagrams corresponding to presentations of Boolean reflection monoids have frozen vertices.

\begin{table}[H]
\resizebox{0.8\width}{0.8\height}{
\begin{tabular}{c c c c}
\hline %
Type of $\Phi$ & Boolean reflection monoids & Generators & $\Delta=\Phi^{\varepsilon}$ \\
\hline %
$A_{n-1} (n\geq 2)$ & $M(A_{n-1},\mathcal{B})$ & $\{s_1,\ldots,s_{n-1},s_\varepsilon\}$ &
$\substack{\\ \\ \\
\begin{tikzpicture}
\draw (0,0) circle (2pt);
\node [below] at (0,0) {$1$};
\draw (0.075,0)--(1.925,0);
\draw (2,0) circle (2pt);
\draw[dashed] (2.075,0)--(3.925,0);
\node [below] at (2,0) {$2$};
\draw (4,0) circle (2pt);
\draw (4.075,0)--(5.925,0);
\node [below] at (4,0) {$n-1$};
\draw[fill] (6,0) circle (2pt);
\node [below] at (6,0) {$\varepsilon$};
\end{tikzpicture}}$ \\
%\hline %
$B_n (n\geq 2)$ & $M(B_{n},\mathcal{B})$ & $\{s_0,\ldots,s_{n-1},s_\varepsilon\}$ &
$\substack{\\ \\  \\
\begin{tikzpicture}
\draw (0,0) circle (2pt);
\node [below] at (0,0) {$0$};
\draw (0.075,0)--(1.925,0);
\draw (2,0) circle (2pt);
\node [above] at (1,0) {$2$};
\draw[dashed] (2.075,0)--(3.925,0);
\node [below] at (2,0) {$1$};
\draw (4,0) circle (2pt);
\draw (4.075,0)--(5.925,0);
\node [below] at (4,0) {$n-1$};
\draw[fill] (6,0) circle (2pt);
\node [below] at (6,0) {$\varepsilon$};
\end{tikzpicture}}$ \\
%\hline %
$D_n (n\geq 4)$ & $M(D_{n},\mathcal{B})$ & $\{s_0,\ldots,s_{n-1},s_\varepsilon\}$ &
$\substack{ \\ \\ \\
\begin{tikzpicture}
\draw (0,-1) circle (2pt);
\node [below] at (0,-1) {$0$};
\draw (0.05,-0.95)--(1.95,-0.05);
\draw (0,1) circle (2pt);
\node [above] at (0,1) {$1$};
\draw (0.05,0.95)--(1.95,0.05);
\draw (2,0) circle (2pt);
\draw[dashed] (2.075,0)--(3.925,0);
\node [below] at (2,0) {$2$};
\draw (4,0) circle (2pt);
\draw (4.075,0)--(5.925,0);
\node [below] at (4,0) {$n-1$};
\draw[fill] (6,0) circle (2pt);
\node [below] at (6,0) {$\varepsilon$};
\end{tikzpicture}}$ \\
\hline %
\end{tabular}}
\caption{Boolean reflection monoids and Dynkin diagrams $A^\varepsilon_{n-1},B^\varepsilon_n,D^\varepsilon_n$.}\label{initial ABD}
\end{table}

In Proposition 3.1 of \cite{EF10}, Everitt and Fountain proved that the symmetric inverse
semigroup $\mathscr{I}_{n}$ is isomorphic to the Boolean reflection monoid of type $A_{n-1}$. We recover the presentation of the symmetric inverse semigroup $\mathscr{I}_{n}$ defined in \cite{E11,P61}. The presentation corresponds exactly to the presentation determined by $A^{\varepsilon}_{n-1}$. Moreover, we recover Everitt and Fountain's presentations of Boolean reflection monoids defined in Section 3 of \cite{EF13}, see Theorem \ref{Everitt and Fountain's presentations}. These presentations can be obtained from any $\Delta$ quiver by a finite sequence of mutations.

We show in Theorem \ref{inner automorphisms} that the inner automorphism group of the Boolean reflection monoid $M(\Phi,\mathcal{B})$ is naturally isomorphic to $W(\Phi)/Z(W(\Phi))$. Surprisingly, inner by diagram automorphisms of irreducible Weyl groups and Boolean reflection monoids can be constructed by a sequence of mutations preserving the same underlying diagrams, see Theorem \ref{mutate inner automorphisms of weyl groups} and Theorem \ref{mutate inner automorphisms} respectively.

Finally we study the cellularity of the semigroup algebras of Boolean reflection monoids. It is well known that Hecke algebras of finite type, $q$-Schur algebras, the Brauer algebra, the Temperley-Lieb algebras and partition algebras are cellular, see \cite{G06,G07,GL96,Xi99,Xi06}. Recently, the cellularity of semigroup algebras is investigated by East \cite{E06},  Wilox \cite{W07}, Guo and Xi \cite{GXi09}, and Ji and Luo \cite{JL16} respectively. Applying Geck's and East's results, we show that semigroup algebras of Boolean reflection monoids are cellular algebras, see Proposition \ref{cellularity of Boolean reflection monoid}. Moreover, we describe their cellular bases using the presentations we obtained.

Automorphisms of finite symmetric inverse semigroups are studied in \cite{ST97}. Coxeter arrangement monoids are defined in \cite{EF10,EF13} and studied in \cite{EF10,EF13,F03,FJ98}. Braid inverse monoids are defined in \cite{EL04} and studied in \cite{EL04,E07}. It would be interesting to study automorphisms of Boolean reflection monoids, presentations of Coxeter arrangement monoids and braid inverse monoids using the method in this paper. We plan to study these in future publications.

The paper is organized as follows. In Section \ref{preliminaries}, we recall some notions and background knowledge that will be used later. In Section \ref{reflection groups}, we recall Barot and Marsh's work about presentations of irreducible Weyl groups and then study inner by diagram automorphisms of irreducible Weyl groups (Theorem \ref{mutate inner automorphisms of weyl groups}). In Section \ref{main results}, we state our main results, Theorem \ref{mutation-invariance} and Theorem \ref{Everitt and Fountain's presentations}, which show that presentations of Boolean reflection monoids are compatible with quiver mutations. In Section \ref{diagrams of finite types}, we describe mutations of $\Delta$ quivers and the oriented cycles appearing in the mutations. In Section \ref{cycle relations and path relations}, we define the inverse monoid $M(Q)$. In Section \ref{the proof of main theorem}, we prove Theorem \ref{mutation-invariance}.

\section{Preliminaries} \label{preliminaries}

\subsection{Mutation of quivers}
Let $Q$ be a quiver with finitely many vertices and finitely many arrows that have no loops or oriented 2-cycles. Given a quiver $Q$, let $I$ be the set of its vertices and $Q^{op}$ its opposite quiver with the same set of vertices but with the reversed orientation for all the arrows. We consider a quiver with weight associated to a skew-symmetrizable matrix, see \cite{FZ03}. 

For each mutable vertex $k$ of $Q$, one can define a mutation of $Q$ at $k$, \cite{BGP73,FZ03}. This produces a new quiver denoted by $\mu_k(Q)$ which can be obtained from $Q$ in the following way \cite{FZ03}:
\begin{enumerate}
\item[(i)] The orientations of all edges incident to $k$ are reversed and their weights intact.

\item[(ii)] For any vertices $i$ and $j$ which are connected in $Q$ via a two-edge oriented path going through $k$, the quiver mutation $\mu_k$ affects the edge connecting $i$ and $j$ in the way shown in Figure \ref{mutation diagram}, where the weights $c$ and $c'$ are related by
    \[
    \pm\sqrt{c}  \pm\sqrt{c'}  = \sqrt{ab},
    \]
    where the sign before $\sqrt{c}$ (resp., $\sqrt{c'}$) is ``$+$" if $i$, $j$, $k$ form an oriented cycle in $Q$ (resp., in $\mu_k(Q)$), and is ``$-$" otherwise. Here either $c$ or $c'$ may be equal to $0$, which means no arrows between $i$ and $j$.
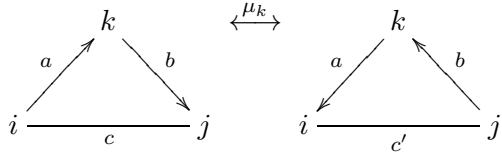
\begin{figure}[H]
\centerline{
$\xymatrix{ & k \ar@{->}[dr]^{b} & \\
i \ar@{->}[ur]^{a} &  & j \ar@{-}[ll]^{c}} \overset{\mu_{k}}{\longleftrightarrow}
\xymatrix{ & k \ar@{->}[dl]_{a} & \\
i \ar@{-}[rr]_{c'} &  & j \ar@{->}[ul]_{b}}$}
\caption{Quiver mutation}\label{mutation diagram}
\end{figure}

\item[(iii)] The rest of the edges and their weights in $Q$ remain unchanged.
\end{enumerate}

Two quivers $Q_1$ and $Q_2$ are said to be \textit{mutation equivalent} if there exists a finite sequence of mutations taking one to the other. We write $Q_1 \sim_{\text{mut}} Q_2$ to indicate that $Q_1$ is mutation equivalent to $Q_2$ and $\mathcal{U}(Q)$ for the mutation class of a quiver $Q$. 

The underlying diagram of a quiver $Q$ is a undirected diagram obtained from $Q$ by forgetting the orientation of all the arrows. Dynkin quivers are quivers whose underlying diagrams are Dynkin diagrams of finite type. A quiver $Q$ is called \textit{connected} if its underlying diagram is connected (every node is ``reachable"). It was shown in {\cite[Theorem 1.4]{FZ03}} that there are only finitely many quivers in the mutation classes of Dynkin quivers. A cycle in the underlying diagram of a quiver is called a \textit{chordless cycle} if no two vertices of the cycle are connected by an edge that does not belong to itself, see \cite{BM15}. As shown in {\cite[Proposition 9.7]{FZ03}} (or see {\cite[Proposition 2.1]{BM15}}), all chordless cycles are oriented in the mutation classes of Dynkin quivers.

\subsection{Cellular algebras and cellular semigroups}
Let us first recall the basic definition of cellular algebras introduced by Graham and Lehrer
\cite{GL96}.

Let $R$ be a commutative ring with identity.

\begin{definition}
An associative $R$-algebra $A$ is called a cellular
algebra with cell datum $(\Lambda,M,C,i)$ if the following conditions are satisfied:
\begin{itemize}
  \item[(C1)] $\Lambda$ is a finite partially ordered set. Associated with each $\lambda \in \Lambda$ there is a finite set $M(\lambda)$ of indices and there exists an $R$-basis $\{ C_{S,T}^{\lambda}\mid \lambda\in \Lambda; S,T\in M(\lambda)\}$ of $A$;
  \item[(C2)] $i$ is an $R$-linear anti-automorphism of $A$ with $i^2 = i$, which sends $C_{S,T}^{\lambda}$ to $C_{T,S}^{\lambda}$;
  \item[(C3)] For each $\lambda \in \Lambda$, $S,T \in M(\lambda)$, and each $a \in A$, \[
      a C_{S,T}^{\lambda} \equiv \sum_{S'\in M(\lambda)} r_a(S',S)C_{S',T}^{\lambda} \pmod{A(< \lambda)},
      \]
      where $r_a(S',S) \in R$ is independent of $T$ and $A (< \lambda)$ is the $R$-submodule of $A$
      generated by $\{C^{\mu}_{S'',T''}\mid \mu < \lambda, S'', T'' \in M(\mu)\}$.
\end{itemize}
\end{definition}

Cellular algebras provide a general framework for studying the representation theory of many important classes of algebras including Hecke algebras of finite type, $q$-Schur algebras, the Brauer algebra, the Temperley-Lieb algebras and partition algebras, see \cite{G06,G07,GL96,Xi99,Xi06}.

Recently, the cellularity of semigroup algebras is investigated by East \cite{E06},  Wilox \cite{W07}, Guo and Xi \cite{GXi09}, and Ji and Luo \cite{JL16} respectively. In the following, we recall some basic notions and facts from the theory of semigroups.

Let $S$ be a semigroup. For any $a, b\in S$, define
\[
a\ \mathcal{L} \ b  \Leftrightarrow   S^1a=S^1b, \quad
a\ \mathcal{R} \ b  \Leftrightarrow   aS^1=bS^1, \quad
a\ \mathcal{J} \ b  \Leftrightarrow   S^1aS^1=S^1bS^1,
\]
$\mathcal{H} = \mathcal{L}\cap \mathcal{R}$ and $\mathcal{D} = \mathcal{L}\vee \mathcal{R}=\mathcal{L}
\circ \mathcal{R}=\mathcal{R}\circ \mathcal{L}$, where $S^1$ is the monoid obtained from $S$ by adding an identity if necessary.  A semigroup $S$ is said to be \textit{inverse} if for each element $s \in S$, there exists a unique inverse $s^{-1} \in S$ such that $ss^{-1}s=s$ and $s^{-1}ss^{-1}=s^{-1}$. If $S$ is a finite inverse semigroup, then $\mathcal{D} = \mathcal{J}$. For $a$ in a semigroup $S$, denote by $H_a$ (respectively, $L_a$, $R_a$, $D_a$, $J_a$) the $\mathcal{H}$ (respectively, $\mathcal{L}$, $\mathcal{R}$, $\mathcal{D}$, $\mathcal{J}$)-class of $a$. For any $s, t\in S$, define $D_s \leq D_t$ if and only if $s \in S^1tS^1$.

Let $S$ be an inverse semigroup with the set $E(S)$ of idempotents. Suppose that $e_1,\ldots,e_k \in D \cap E(S)$. Choose $a_1=e_1, \ldots,a_k \in L_{e_1}$ such that $a_j\;R\;e_j$ for each $j$. Then $D = R_{e_1}\cup R_{a_2} \cup \dots \cup R_{a_k}$. Put $H_D = H_{e_1}$ and by Green's Lemma, for each $x\in D$ we have $x = a_i g a^{-1}_j$ for unique $1 \leq i,j\leq k$ and $g\in H_D$. Using East's symbol in \cite{E06}, let $[e_i, e_j, g]_D$ be the element $x$. Then $x^{-1}=[e_j, e_i, g^{-1}]_D$. For more detailed knowledge about semigroups, the reader is referred to \cite{E06,H95}.

A semigroup $S$ is said to be \textit{cellular} if its semigroup algebra $R[S]$ is a cellular algebra. In \cite{E06}, East proved the following theorem:

\begin{theorem}[\protect{\cite[Theorem 15]{E06}}]\label{East's results}
Let $S$ be a finite inverse semigroup with the set $E(S)$ of idempotents. If $S$ satisfies
the following conditions:
\begin{itemize}
  \item[(1)] For each $\mathcal{D}$-class $D$, the subgroup $H_D$ is cellular with cell datum $(\Lambda_D,M_D,C_D,i_D)$;
  \item[(2)] The map $i: R[S] \to R[S]$ sending $[e,f,g]_D$ to $[f,e,i_D(g)]_D$ is an $R$-linear anti-homomorphism.
\end{itemize}
Then $S$ is a cellular semigroup with cell datum $(\Lambda,M,C,i)$, where $\Lambda = \{(D,\lambda) \mid D \in S/\mathcal{D}, \lambda \in \Lambda_{D} \}$ with partial order defined by
$(D,\lambda) \leq (D',\lambda')$ if $D < D' \text{ or } D = D' \text{ and } \lambda \leq \lambda'$, $M(D,\lambda) = \{ (e,s)\mid e\in E(S)\cap D, s \in M_{D}(\lambda) \}$ for $(D,\lambda) \in \Lambda$, and $C = \{ C^{(D,\lambda)}_{(e,s),(f,t)}=[e,f,C^{\lambda}_{s,t}]_D \mid (D,\lambda) \in \Lambda; (e,s), (f,t) \in M(D,\lambda)\}$ for $(D,\lambda) \in \Lambda$ and $(e,s), (f,t) \in M(D,\lambda)$.
\end{theorem}

By the definition of cellular algebras, we have the following corollary.
\begin{corollary}\label{new cellular bases}
Suppose that $A$ is a cellular algebra over $R$ with cell datum $(\Lambda,M,C,i)$ and $\overline{\varphi}$ is an $R$-linear automorphism of $A$. Let $\overline{C}^{\lambda}_{S,T}=\overline{\varphi}(C^{\lambda}_{S,T})$ for any $\lambda \in \Lambda$, $(S,T) \in M(\lambda) \times M(\lambda)$,
and $\overline{i} = \overline{\varphi} i \overline{\varphi}^{-1}$. Then $(\Lambda,M,\overline{C},\overline{i})$ is a cellular datum of $A$.
\end{corollary}
\begin{proof}
Since $\{C^{\lambda}_{S,T}\mid \lambda \in \Lambda, (S,T) \in M(\lambda) \times M(\lambda)\}$ is an $R$-basis of $A$ and $\overline{\varphi}$ is an $R$-linear automorphism of $A$, we have the fact that $\{\overline{C}^{\lambda}_{S,T} \mid \lambda \in \Lambda, (S,T) \in M(\lambda) \times M(\lambda)\}$ is also an $R$-basis of $A$. It follows from the definition of $\overline{i}$ that $\overline{i}^2=\overline{i}$ and  $\overline{i}(\overline{C}^{\lambda}_{S,T})=\overline{C}^{\lambda}_{T,S}$.

For each $\lambda \in \Lambda$, $S,T \in M(\lambda)$, and each $a \in A$,
\[
a C_{S,T}^{\lambda} \equiv \sum_{S'\in M(\lambda)} r_a(S',S)C_{S',T}^{\lambda} \pmod{A(< \lambda)},
\]
where $r_a(S',S) \in R$ is independent of $T$ and $A (< \lambda)$ is the $R$-submodule of $A$ generated by $\{C^{\mu}_{S'',T''}\mid \mu < \lambda, S'', T'' \in M(\mu)\}$. For each $a \in A$, there exists a $b \in A$ such that $a=\overline{\varphi}(b)$. Then
\begin{align*}
a \overline{C}^{\lambda}_{S,T} = \overline{\varphi}(b) \overline{\varphi}(C_{S,T}^{\lambda}) = \overline{\varphi}(b C_{S,T}^{\lambda}) & \equiv \sum_{S'\in M(\lambda)} r_b(S',S)
\overline{\varphi}(C_{S',T}^{\lambda}) \pmod{A(< \lambda)} \\
&  \equiv \sum_{S'\in M(\lambda)} r_b(S',S)
\overline{C}^{\lambda}_{S',T} \pmod{A(< \lambda)},
\end{align*}
where $r_b(S',S) \in R$ is independent of $T$. Therefore $(\Lambda,M,\overline{C},\overline{i})$ is a cellular basis of $A$, as required.
\end{proof}

\section{Some new results of irreducible Weyl groups}\label{reflection groups}
In this section, we recall Barot and Marsh's work about presentations of irreducible Weyl groups and then study inner by diagram automorphisms of irreducible Weyl groups.
We refer the reader to \cite{B02, FR08, H90, M13} for more information about Weyl groups, root systems, and reflection groups.

\subsection{Barot and Marsh's results}
It is well known that finite irreducible crystallographic reflection groups or irreducible Weyl groups are classified by Dynkin diagrams, see \cite{H90}. Let $\Gamma$ be a Dynkin diagram and $I$ the set of its vertices. Let $W_\Gamma$ be the finite irreducible Weyl group determined by $\Gamma$. A quiver is called a $\Gamma$ quiver if its underlying diagram is $\Gamma$. In \cite{BM15}, Barot and Marsh gave presentations of $W_\Gamma$. Their construction works as follows. Let $Q$ be a quiver that is mutation equivalent to a $\Gamma$ quiver. For two vertices $i$, $j$ of $Q$, one defines
\begin{align}\label{definition of M}
m_{ij}=\begin{cases}
2 & \textrm{$i$ and $j$ are not connected},\\
3 & \textrm{$i$ and $j$ are connected by an edge with weight $1$},\\
4 & \textrm{$i$ and $j$ are connected by an edge with weight $2$},\\
6 & \textrm{$i$ and $j$ are connected by an edge with weight $3$}.
\end{cases}
\end{align}

\begin{definition}\label{definition of groups}
Let $W(Q)$ be the group with generators $s_i$, $i\in I$, subjecting to the following relations:
\begin{itemize}
  \item[(1)] $s^{2}_{i}=e$ for all $i$;
  \item[(2)] $(s_is_j)^{m_{ij}}=e$ for all $i \neq j$;
  \item[(3)] For any chordless cycle $C$ in $Q$:
  \begin{align*}
  i_0 \overset{\omega_1}{\longrightarrow} i_1 \overset{\omega_2}{\longrightarrow} \cdots \overset{\omega_{d-1}}{\longrightarrow} i_{d-1}
  \overset{\omega_0}{\longrightarrow} i_0,
  \end{align*}
  where either all of the weights are $1$, or $\omega_0=2$, we have:
  $$(s_{i_0}s_{i_1}\cdots s_{i_{d-2}}s_{i_{d-1}}s_{i_{d-2}}\cdots s_{i_1})^2=e;$$
\end{itemize}
where $e$ is the identity element of $W(Q)$.
\end{definition}

One of Barot and Marsh's main results in \cite{BM15} is stated as follows.
\begin{theorem}[{\cite[Theorem A]{BM15}}]\label{Mash's result}
The group $W(Q)$ does not depend on the choice of a quiver in $\mathcal{U}(Q)$. In particular, $W(Q) \cong W_\Gamma$ for each quiver $Q$ mutation equivalent to a $\Gamma$ quiver.
\end{theorem}

\subsection{Inner by diagram automorphisms of irreducible Weyl groups}
Let $W$ be a Coxeter group defined by a set $S$ of generators and relations (1) and (2) of Definition \ref{definition of groups}. The pair $(W,S)$ is called a Coxeter system of $W$. In what follows, given any two Coxeter systems $(W, S_1)$ and $(W, S_2)$, when we say that there exists an automorphism of $W$, we mean that there is an automorphism $\alpha\in \text{Aut}(W)$ such that $\alpha(S_1) = S_2$. If such an automorphism can always be chosen from $\text{Inn}(W)$, the group of inner automorphisms of $W$, then $W$ is called \textit{strongly rigid}. In case $W$ is strongly rigid, the group $\text{Aut}(W)$ has a very simple structure (see Corollary 3.2 of \cite{CD00}):
\begin{align*}
\text{Aut}(W)=\text{Inn}(W) \times \text{Diag}(W),
\end{align*}
where $\text{Diag}(W)$ consists of diagram automorphisms of the unique Coxeter diagram corresponding to $W$.

The following lemma is well known.
\begin{lemma}\label{the form of reflections}
Let $W$ be a finite group generated by a finite set $S$ of simple reflections. Then the set of all reflections in $W$ is $\{wsw^{-1}\mid w \in W, s \in S\}$.
\end{lemma}

In \cite[Table I]{B69}, Bannai computed the center $Z(W(\Phi))$ of an irreducible Weyl group $W(\Phi)$. The longest element $w_0$ in $W(\Phi)$ is a central element of $W(\Phi)$ except for $\Phi=A_{n} (n\geq 2),\; D_{2k+1} (k\geq 2),\; E_{6}$.

The following important notion was introduced by Franzsen in \cite{Fran01}.
\begin{definition}[{\cite[Definition 1.36]{Fran01}}]
An inner by diagram automorphism is an automorphism generated by some inner and diagram automorphisms in $\text{Aut}(W)$. The subgroup of inner automorphisms is a normal subgroup of $\text{Aut}(W)$, therefore any inner by diagram automorphism can be written as the product of an inner and a diagram automorphism.
\end{definition}

By Proposition 1.44 of \cite{Fran01}, we know that all automorphisms of irreducible Weyl groups that preserve reflections are inner by diagram automorphisms.

%The next lemma gives some facts from \cite{Fran01} which will be used later.
%\begin{lemma}[{\cite[Proposition 2.4]{Fran01}}]\label{automorphism of A}
%If $W$ is Weyl group of type $A_n$, then $\text{Aut}(W(A_n))\cong W(A_n)$ if $n\neq 5$, while $\text{Aut}(W(A_5)) \cong W(A_5) \rtimes \mathbb{Z}_2$. So, for $n\neq 5$, any automorphism of Weyl group of type $A_n$ maps reflections to reflections, furthermore any automorphism of $A_5$ that does preserve reflection is inner.
%\end{lemma}

%\begin{lemma}[{\cite[Proposition 1.44, Propositions 2.8--2.10, Proposition 2.13]{Fran01}}] \label{automorphism of BDEFG}
%Let $W(\Phi)$ be Weyl group of type $\Phi$.
%\begin{itemize}
  %\item[(a)] Any automorphism of $W(B_n)$ for $n > 2$ that does preserve reflections must be inner.
  %\item[(b)] For $k \geq 2$, $\text{Aut}(W(D_{2k+1}))\cong W(D_{2k+1})$, that is, all automorphisms of $W(D_{2k+1})$ map reflections to reflections.
  %\item[(c)] All automorphisms of $W(E_6)$ or $W(E_7)$ are inner.
  %\item[(d)] All automorphisms of Weyl groups that preserve reflections are inner by diagram automorphisms.
%\end{itemize}
%\end{lemma}

The following theorem reveals a connection between inner by diagram automorphisms of irreducible Weyl groups and quiver mutations.

\begin{theorem}\label{mutate inner automorphisms of weyl groups}
Let $Q$ be a $\Gamma$ quiver and $W(Q)$ the corresponding Weyl group generated by a set $S$ of simple reflections. Then $\alpha$ is an inner by diagram automorphism of $W(Q)$ if and only if
there exists a sequence of mutations preserving the underlying diagram $\Gamma$ such that $\alpha(S)$ can be obtained from $Q$ by mutations. In particular, all reflections in $W(Q)$ can be obtained from $Q$ by mutations.
\end{theorem}

\begin{proof}
By observation, every variable obtained by mutations must be some reflection of the corresponding Weyl group. The sufficiency follows from the fact that all automorphisms of Weyl groups that preserve reflections are inner by diagram automorphisms.

To prove necessity, assume without loss of generality that the vertex set of $Q$ is $\{1,2,\ldots,n\}$ and $\alpha$ is an inner by diagram automorphism of $W(Q)$. Note that diagram automorphisms of any Dynkin diagram keep the underlying Dynkin diagram. Then relabelling the vertices of $Q$ if necessary, there exists an inner automorphism $\overline{\alpha}$ of $W(Q)$ such that $\overline{\alpha}(S)=\alpha(S)$. It is sufficient to prove that $\overline{\alpha}(S)$ can be obtained from $Q$ by mutations and the sequence of mutations preserves the underlying diagram $\Gamma$.

Let $g=s_{i_1} s_{i_2} \cdots s_{i_r} \in W(Q)$ be a reduced expression for $g$, where $s_{i_k} \in S, 1 \leq k \leq r$. We assume that $\overline{\alpha}(S)=gSg^{-1}=\{gsg^{-1}\mid s\in S\}$. In the following we use induction to construct a sequence of mutations preserving the underlying diagram $\Gamma$ such that $gSg^{-1}$ is obtained from $Q$ by mutations.

{\bf Step 1.} We mutate firstly $Q$ at the vertex $i_1$ twice. Then we get a quiver $Q_1$, which has the same underlying diagram with $Q$. Moreover, the set $S$ becomes
\[
s_{i_1} S s_{i_1} =\{s_{i_1}s_1s_{i_1}, s_{i_1}s_2s_{i_1}, \ldots, s_{i_1}s_{n-1}s_{i_1}, s_{i_1} s_n s_{i_1} \}.
\]
We keep vertices of $Q_1$ having the same label as vertices of $Q$.

%{\bf Step 2.} We then mutate $Q_1$ at the vertex $i_2$ twice. Note that the label labeled in vertex $i_2$ of $Q_1$ is $s_{i_1} s_{i_2}s_{i_1}$. Then we get a quiver $Q_2$, which has the same underlying diagrams with $Q, Q_1$. Moreover, the set of vertices in $Q_2$ becomes $$s_{i_1} s_{i_2} s_{i_1} s_{i_1} S s_{i_1} s_{i_1}s_{i_2}s_{i_1} = s_{i_1}s_{i_2} S s_{i_2}s_{i_1}.$$ We keep the labels labeled in vertices of $Q_2$ the same as $Q, Q_1$.

{\bf Step 2.} We then mutate $Q_{t-1}$, $t=2,3,\cdots, r$ at the vertex $i_t$ twice. Note that the variable corresponding to the vertex $i_t$ of $Q_{t-1}$ is $s_{i_1}\ldots s_{i_{t-1}} s_{i_t}s_{i_{t-1}}\ldots s_{i_1}$. Then we get a quiver $Q_{t}$, which has the same underlying diagram with $Q, Q_1,\ldots, Q_{t-1}$. Moreover, the set $s_{i_1} \ldots s_{i_{t-1}} S s_{i_{t-1}}\ldots  s_{i_1}$ of generators in $Q_{t-1}$ becomes
\begin{align*}
& s_{i_1}\cdots s_{i_{t-1}} s_{i_t}s_{i_{t-1}}\cdots s_{i_1} (s_{i_1} \cdots s_{i_{t-1}} S s_{i_{t-1}}\cdots  s_{i_1} ) s_{i_1}\cdots s_{i_{t-1}} s_{i_t}s_{i_{t-1}}\cdots s_{i_1} \\
& = s_{i_1}\cdots s_{i_{t-1}} s_{i_t} S s_{i_t}s_{i_{t-1}}\cdots s_{i_1}.
\end{align*}
We keep vertices of $Q_t$ having the same label as vertices of $Q_{t-1}$.

{\bf Step 3.} We repeat Step 2 until we get the quiver $Q_r$.

By induction, it is not difficult to show that the set of generators in $Q_r$ is
\[
s_{i_1} s_{i_2} \cdots s_{i_r} S s_{i_r}\cdots s_{i_2} s_{i_1}=g S g^{-1}=\overline{\alpha}(S).
\]

Finally, every reflection in $W(Q)$ is conjugate to a simple reflection by Lemma \ref{the form of reflections}. So we assume that every reflection is of the form $gs_ig^{-1}$, where  $g=s_{i_1}s_{i_2}\cdots s_{i_k}$ is a reduced expression for $g$ in $W(Q)$. By the same arguments as before, we mutate the sequence $i_1, i_1, i_2, i_2, \ldots, i_k,i_k$ starting from $Q$, we obtain the reflection $gs_ig^{-1}$.
\end{proof}
\begin{remark}
\begin{itemize}
  \item[(1)] In types $A_n$, $B_n (n>2)$, $D_{2k+1}$, $E_6$, $E_7$, and $E_8$, all inner by diagram automorphisms of the corresponding Weyl groups are inner automorphisms. $W(\Phi)$ is strongly rigid for $\Phi=A_n (n\neq 5), D_{2k+1}, E_6, E_7$.
  \item[(2)] If $(W,S)$ is a Coxeter system of $W$, then for any inner by diagram automorphism $\alpha$ of $W$, we have that $(W,\alpha(S))$ is also a Coxeter system of $W$.
  \item[(3)] Suppose that $Q$ is a quiver that is mutation equivalent to a $\Gamma$ quiver. Let $W(Q)$ be the corresponding Weyl group with a set $S$ of generators. Then Theorem \ref{mutate inner automorphisms of weyl groups} holds for $Q$.

\end{itemize}
\end{remark}

\section{Main results on Boolean reflection monoids}\label{main results}
Let $A^\varepsilon_{n-1}$ (respectively, $B^\varepsilon_n$, $D^\varepsilon_n$) be the Dynkin diagram with $n$ (respectively, $n+1$, $n+1$) vertices, among them the first $n-1$ (respectively, $n$, $n$) vertices are mutable vertices and the last vertex $\varepsilon$ is a frozen vertex, as is shown in Table \ref{initial ABD}. The label is left on an edge only if its weight is greater or equal to 2, otherwise left unlabelled.

In this section, we assume that $\Delta$ is one of $A^\varepsilon_{n-1}$, $B^\varepsilon_n$ and $D^\varepsilon_n$. A quiver is said to be a \textit{$\Delta$ quiver} if the underlying diagram of such quiver is $\Delta$. It follows from the connectivity and finiteness of $\Delta$ that any two $\Delta$ quivers are mutation equivalent. Let $\mathcal{U}(\Delta)$ be the mutation class of any $\Delta$ quiver.

\subsection{Boolean reflection monoids}
In 2010, Everitt and Fountain introduced reflection monoids, see \cite{EF10}. The Boolean
reflection monoids are a family of reflection monoids (symmetric inverse semigroups are Boolean reflection monoids of type $A$).

Let $V$ be a Euclidean space with standard orthonormal basis $\{v_1, v_2, \ldots, v_n\}$ and let $\Phi \subseteq V$ be a root system and $W(\Phi)$ the associated Weyl group of type $\Phi$. A partial linear isomorphism of $V$ is an isomorphism between two subspaces of $V$. Any partial linear isomorphism of $V$ can be realized by restricting an isomorphism of $V$ to some subspace. Let $g,h$ be two isomorphisms of $V$ and $Y,Z$ be two subspaces of $V$. We denote by $g_Y$ the partial isomorphism whose domain is $Y$ and whose range is the restriction of $g$ to $Y$. We denote by $M(V)$ (respectively, $GL(V)$) the \textit{general linear monoid} (respectively, \textit{general linear group}) on $V$ consisting of partial linear isomorphisms (respectively, linear isomorphisms) of $V$. In $M(V)$, we have $g_Y = h_Z$ if and only if $Y=Z$ and $gh^{-1}$ is in the isotropy group $G_Y=\{g\in GL(V)\;|\;gv=v \text{ for all } v\in Y\}$. Moreover, $g_Y h_Z=(gh)_{Z \cap h^{-1}Y}$ and $(g_Y)^{-1}=(g^{-1})_{gY}$.

Let us recall the notion ``system" in $V$ for a group $G\subseteq GL(V)$  introduced in \cite{EF10}.
\begin{definition}[{\cite[Definition 2.1]{EF10}}]
Let $V$ be a real vector space and $G\subseteq GL(V)$ a group. A collection $\mathcal{S}$ of subspaces of $V$ is called a \textit{system} in $V$ for $G$ if and only if
\begin{itemize}
  \item[(1)] $V \in \mathcal{S}$,
  \item[(2)] $G\mathcal{S}=\mathcal{S}$, that is, $gX \in \mathcal{S}$ for any $g \in G$ and $X \in
  \mathcal{S}$,
  \item[(3)] if $X, Y \in \mathcal{S}$, then $X \cap Y \in \mathcal{S}$.
\end{itemize}
\end{definition}

For $J\subseteq X=\{1,2, \ldots, n\}$, let $X(J)=\bigoplus_{j \in J} \mathbb{R}v_{j} \subseteq V$, and $\mathcal{B}=\{X(J): J \subseteq X\}$ with $X(\emptyset)={\bf 0}$. Then $\mathcal{B}$ is a Boolean system in $V$ for $W(\Phi)$, where $\Phi= A_{n-1}$, $B_n/C_n$, or $D_n$. For example, the Weyl group $W(A_{n-1})$-action on the subspaces $X(J)\in \mathcal{B}$ is just $g(\pi)X(J)=X(\pi J)$, where $\pi \mapsto g(\pi)$ induces an isomorphism from the symmetric group $S_n$ on the set $X$ to $W(A_{n-1})$, $\pi J=\{\pi(j)\mid j\in J\}$. Note that $\mathcal{B}$ is not a system for any of the exceptional $W(\Phi)$.

%Root systems of types $B_n$ and $C_n$ give rise to the same Weyl group. So we only concern the Weyl group $W(B_n)$.

\begin{definition}[{\cite[Definition 2.2]{EF10}}]
Let $G\subseteq GL(V)$ be a group and $\mathcal{S}$ a system in $V$ for $G$. The monoid of partial linear isomorphisms given by $G$ and $\mathcal{S}$ is the submonoid of $M(V)$ defined by
\begin{align*}
M(G,\mathcal{S}):=\{g_X : g \in G, \ X \in \mathcal{S} \}.
\end{align*}
If $G$ is a reflection group, then $M(G,\mathcal{S})$ is called a \textrm{reflection monoid}.
\end{definition}

Let $G$ be the reflection group $W(\Phi)$ for $\Phi= A_{n-1}$, $B_n/C_n$, or $D_n$, and $\mathcal{S}$ the Boolean system in $V$ for $G$, then $M(G,\mathcal{S})$ is called a \textit{Boolean reflection monoid}. In general, we write $M(\Phi,\mathcal{B})$ instead of $M(W(\Phi),\mathcal{B})$, and call $M(\Phi,\mathcal{B})$ the Boolean reflection monoid of type $\Phi$. 

We recall Everitt and Fountain's presentations in Section 4 of \cite{EF13}:

\begin{lemma}\label{EF presentations}
Everitt and Fountain's presentations of Boolean reflection monoids are as follows:
\begin{align*}
M(A_{n-1},\mathcal{B})=& \left\langle  s_1,s_2, \ldots, s_{n-1}, s_{\varepsilon} \mid
(s_is_j)^{m_{ij}}=e\ \text{for } 1 \leq i,j\leq n-1, \right. \\
& \left. \ s_{\varepsilon}^2=s_{\varepsilon}, s_i s_{\varepsilon} = s_{\varepsilon} s_i \text{ for $i\neq 1$}, \right. \\
& \left. \ s_{\varepsilon} s_{1} s_{\varepsilon} s_{1}= s_{1} s_{\varepsilon} s_{1} s_{\varepsilon} =s_{\varepsilon} s_{1} s_{\varepsilon} \right\rangle, \\
M(B_{n},\mathcal{B})=& \left\langle s_0,s_1,\ldots,s_{n-1}, s_{\varepsilon} \mid (s_is_j)^{m_{ij}}=e \text{ for } 0 \leq i,j\leq n-1, \right. \\
&\left. s_{\varepsilon}^2=s_{\varepsilon},\ s_0 s_{1} s_{\varepsilon} s_{1} = s_{1} s_{\varepsilon}
s_{1} s_{0}, \ s_0 s_{\varepsilon} = s_{\varepsilon}, \right. \\
&\left. s_is_{\varepsilon}=s_{\varepsilon}s_i\text{ for $i\neq 1$},\ s_{1}s_{\varepsilon}s_{1}s_{\varepsilon}=s_{\varepsilon}s_{1}s_{\varepsilon}s_{1}=
s_{\varepsilon}s_{1}s_{\varepsilon}\right.\rangle, \\
M(D_{n},\mathcal{B})=& \left\langle s_0,s_1,\ldots,s_{n-1},s_{\varepsilon} \mid (s_is_j)^{m_{ij}}=e\
\text{for } 0 \leq i,j\leq n-1,\right. \\
&\left. \ s_{\varepsilon}^2=s_{\varepsilon}, s_i s_{\varepsilon} = s_{\varepsilon} s_i \text{ for $i> 1$},\
s_{\varepsilon} s_{1} s_{\varepsilon} s_{1}= s_{1} s_{\varepsilon} s_{1} s_{\varepsilon} = s_{\varepsilon} s_{1} s_{\varepsilon}, \right. \\
&\left. \ s_0 s_{\varepsilon} s_{0}  = s_1 s_{\varepsilon} s_1,\ s_0 s_2 s_1 s_{\varepsilon} s_1 s_2
= s_2 s_1 s_{\varepsilon} s_1 s_2 s_0
\right\rangle,
\end{align*}
where $m_{ij}$ is defined in (\ref{definition of M}).
\end{lemma}

In \cite{EF10}, Everitt and Fountain proved that the Boolean reflection monoid $M(A_{n-1},\mathcal{B})$ (respectively, $M(B_{n},\mathcal{B})$, $M(D_{n},\mathcal{B})$) is isomorphic to the symmetric inverse semigroup $\mathscr{I}_{n}$ (respectively, the monoid $\mathscr{I}_{\pm n}$ of partial signed permutations, the monoid $\mathscr{I}^{e}_{\pm n}$ of partial even signed permutations). 
 
\subsection{Mutation classes of $\Delta$ quivers}\label{mutations classes with frozen vertices}
By a result of Fomin and Zelevinsky \cite{FZ03}, the mutation class of a Dynkin quiver is finite. Let $\mathcal{U}(A_{k+1})$ be the mutation class of Dynkin quivers of type $A_{k+1}$. We use a description of $\mathcal{U}(A_{k+1})$ given in \cite{BV08}:
\begin{itemize}
  \item[(1)] Each quiver has $k+1$ vertices.
  \item[(2)] All nontrivial cycles are oriented and of length 3.
  \item[(3)] A vertex has at most four incident arrows.
  \item[(4)] If a vertex has four incident arrows, then two of them belong to one oriented 3-cycle and the other belong to another oriented 3-cycle.
  \item[(5)] If a vertex has three incident arrows, then two of them belong to one oriented 3-cycle and the third arrow does not belong to any oriented 3-cycles.
\end{itemize}

Let $\mathscr{\phi}(A^{\varepsilon}_{k})$ be the class of quivers which satisfy the above conditions (1)--(5) and
\begin{enumerate}	
	 \item[(6)] The frozen vertex $\varepsilon$ has at most two incident arrows and if it has two incident arrows, then $\varepsilon$ belongs to one oriented 3-cycle.
 \end{enumerate} 

Let $\mathcal{U}^A$ (respectively, $\mathcal{U}^{A^{\varepsilon}}$) be the union of all $\mathcal{U}(A_k)$  (respectively, $\mathscr{\phi}(A^{\varepsilon}_k)$) for all $k$. For a quiver $Q \in \mathcal{U}^A$ or $Q \in \mathcal{U}^{A^{\varepsilon}}$, a vertex $v$ in $Q$ is called a \textit{connecting vertex} if $v$ has at most 2 neighbours and, moreover, if $v$ has 2 neighbours, then $v$ is a vertex in a 3-cycle in $Q$.

Let $\mathscr{\phi}(B^{\varepsilon}_n)$ be the class of quivers which belong to one of the following two types (here $a$ is a connecting vertex for $Q'\in\mathcal{U}^{A^{\varepsilon}}$ and $b$ is a connecting vertex for $Q''\in\mathcal{U}^{A}$):
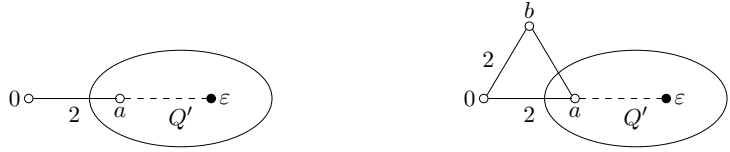
\begin{figure}[H]
\resizebox{0.8\width}{0.8\height}{
\begin{minipage}[b]{0.5\textwidth}
\centerline{
\begin{tikzpicture}
\draw (0,0) circle (2pt);
\draw (1.5,0) circle (2pt);
\draw[fill] (3,0) circle (2pt);
\draw (0.075,0)--(1.45,0);
\draw[dashed] (1.575,0)--(2.925,0);
\draw (2.5,0) ellipse (1.5 cm and 0.7 cm);
\node [left] at (0,0) {$0$};
\node [below] at (1.5,0) {$a$};
\node[right] at (3,0) {$\varepsilon$};
\node [below] at (2.5,0) {$Q'$};
\node [below] at (0.75,0) {$2$};
\end{tikzpicture}}
\end{minipage}
\begin{minipage}[b]{0.5\textwidth}
\centerline{
\begin{tikzpicture}
\draw (0.75,1.2) circle (2pt);
\draw (0,0) circle (2pt);
\draw (1.5,0) circle (2pt);
\draw[fill] (3,0) circle (2pt);
\draw (0.075,0)--(1.45,0);
\draw (0.7,1.15)--(0.05,0.05);
\draw (0.8,1.15)--(1.45,0.05);
\draw[dashed] (1.575,0)--(2.925,0);
\draw (2.5,0) ellipse (1.5 cm and 0.7 cm);
\draw (-1,0) ellipse (1.5 cm and 0.7 cm);
\node [above] at (0.75,1.2) {$0$};
\node [left] at (0,0) {$b$};
\node [below] at (1.5,0) {$a$};
\node[right] at (3,0) {$\varepsilon$};
\node [below] at (2.5,0) {$Q'$};
\node  at (-1,0) {$Q''$};
\node [right] at (1.05,0.7) {$2$};
\node [left] at (0.4,0.7) {$2$};
\end{tikzpicture}}
\end{minipage}}
\caption{$\mathscr{\phi}(B^{\varepsilon}_n)$}\label{mutation class of type B}
\end{figure}

As a generalization of mutation class of Dynkin quivers of type $D_n$ in \cite{V10}, let $\mathscr{\phi}(D^{\varepsilon}_n)$ be the class of quivers which belong to one of the following four types:

Type I: $Q$ has two vertices $a$ and $b$ which have a neighbour and both $a$ and
$b$ have an arrow to or from the same vertex $c$, and $Q'=Q\backslash \{a,b\}$ is in $\mathscr{\phi}(A^{\varepsilon}_{n-2})$, and $c$ is a connecting vertex for $Q'$, see Figure \ref{Type I}.
\begin{figure}[H]
\resizebox{0.75\width}{0.75\height}{
\begin{minipage}[b]{0.5\textwidth}
\centerline{
\begin{tikzpicture}
\draw (1.5,1.5) circle (2pt);
\draw (0,0) circle (2pt);
\draw (1.5,0) circle (2pt);
\draw[fill] (3,0) circle (2pt);
\draw (0.075,0)--(1.45,0);
\draw (1.5,0.05)--(1.5,1.45);
\draw[dashed] (1.575,0)--(2.925,0);
\draw (2.5,0) ellipse (1.6 cm and 0.9 cm);
\node [above] at (1.5,1.5) {$a$};
\node [left] at (0,0) {$b$};
\node [below] at (1.5,0) {$c$};
\node[right] at (3,0) {$\varepsilon$};
\node [below] at (2.5,-0.1) {$Q'$};
\end{tikzpicture}}
\end{minipage}}
\caption{Type I}\label{Type I}
\end{figure}
Type II: $Q$ has a full subquiver with four vertices which looks like the quiver drawn in the left hand side of Figure \ref{Type II} and $Q \backslash \{a,b,c\to d\}=Q' \cup Q''$ is disconnected with two components $Q'\in \mathscr{\phi}(A^{\varepsilon}_{k})$ for some integer $k$ and $Q''\in \mathcal{U}(A_{\ell})$ for some integer $\ell$ and for which $c$ and $d$ are connecting vertices, see the right hand side of Figure \ref{Type II}.
\begin{figure}[H]
\resizebox{0.75\width}{0.75\height}{
\begin{minipage}[b]{0.5\textwidth}
\centerline{
\begin{tikzpicture}[>=latex]
\draw (0,0) circle (2pt);
\draw (2.6,0) circle (2pt);
\draw (1.3,1.8) circle (2pt);
\draw (1.3,-1.8) circle (2pt);
\draw[<-,shorten >=1pt] (0.05,0.05)--(1.25,1.75);
\draw[->,shorten >=1pt] (2.55,0.05)--(1.35,1.75);
\draw[<-,shorten >=1pt] (0.05,-0.05)--(1.25,-1.75);
\draw[->,shorten >=1pt] (2.55,-0.05)--(1.35,-1.75);
\draw[->,shorten >=0.6pt] (0.075,0)--(2.525,0);
\node [above] at (1.3,1.8) {$a$};
\node [left] at (0,0) {$d$};
\node [below] at (1.3,-1.8) {$b$};
\node [right] at (2.6,0) {$c$};
\end{tikzpicture}}
\end{minipage}
\begin{minipage}[b]{0.5\textwidth}
\centerline{
\begin{tikzpicture}[>=latex]
\draw (0,0) circle (2pt);
\draw[fill] (4.5,0) circle (2pt);
\draw (3,0) circle (2pt);
\draw (1.5,2) circle (2pt);
\draw (1.5,-2) circle (2pt);
\draw (0.05,0.05)--(1.45,1.95);
\draw (2.95,0.05)--(1.55,1.95);
\draw (0.05,-0.05)--(1.45,-1.95);
\draw (2.95,-0.05)--(1.55,-1.95);
\draw[dashed] (3.075,0)--(4.425,0);
\draw (0.075,0)--(2.925,0);
\draw (-1,0) ellipse (1.3 cm and 0.8 cm);
\draw (4,0) ellipse (1.3 cm and 0.8 cm);
\node [above] at (1.5,2) {$a$};
\node [left] at (0,0) {$d$};
\node [right] at (4.5,0) {$\varepsilon$};
\node [below] at (3,0) {$c$};
\node [below] at (1.5,-2) {$b$};
\node at (-1.6,0) {$Q''$};
\node[below] at (4,0) {$Q'$};
\end{tikzpicture}}
\end{minipage}}
\caption{Type II}\label{Type II}
\end{figure}
Type III: $Q$ has a full subquiver which is an oriented 4-cycle drawn in the left hand side of Figure \ref{Type III} and $Q \backslash \{a,b\}=Q' \cup Q''$ is disconnected with two components $Q'\in \mathscr{\phi}(A^{\varepsilon}_{k})$ for some integer $k$ and $Q''\in \mathcal{U}(A_{\ell})$ for some integer $\ell$ and for which $c$ and $d$ are connecting vertices, see the right hand side of Figure \ref{Type III}.
\begin{figure}[H]
\resizebox{0.75\width}{0.75\height}{
\begin{minipage}[b]{0.5\textwidth}
\centerline{
\begin{tikzpicture}[>=latex]
\draw (0,0) circle (2pt);
\draw (1.5,0) circle (2pt);
\draw (3,0) circle (2pt);
\draw (1.5,2) circle (2pt);
\draw[<-,shorten >=1pt] (0.05,0.05)--(1.45,1.95);
\draw[->,shorten >=1pt] (2.95,0.05)--(1.55,1.95);
\draw[->,shorten >=1pt] (0.075,0)--(1.425,0);
\draw[->,shorten >=1pt] (1.575,0)--(2.925,0);
\node [above] at (1.5,2) {$a$};
\node [left] at (0,0) {$d$};
\node [below] at (1.5,0) {$b$};
\node [right] at (3,0) {$c$};
\end{tikzpicture}}
\end{minipage}
\begin{minipage}[b]{0.5\textwidth}
\centerline{
\begin{tikzpicture}
\draw (0,0) circle (2pt);
\draw (1.5,0) circle (2pt);
\draw[fill] (4.5,0) circle (2pt);
\draw (3,0) circle (2pt);
\draw (1.5,2.2) circle (2pt);
\draw (0.05,0.05)--(1.45,2.15);
\draw (2.95,0.05)--(1.55,2.15);
\draw (0.075,0)--(1.425,0);
\draw (1.575,0)--(2.925,0);
\draw[dashed] (3.075,0)--(4.425,0);
\draw (-1,0) ellipse (1.3 cm and 0.8 cm);
\draw (4,0) ellipse (1.3 cm and 0.8 cm);
\node [above] at (1.5,2.2) {$a$};
\node [left] at (0,0) {$d$};
\node [below] at (3,0) {$c$};
\node [right] at (4.5,0) {$\varepsilon$};
\node [below] at (1.5,0) {$b$};
\node at (-1.6,0) {$Q''$};
\node[below] at (4,0) {$Q'$};
\end{tikzpicture}}
\end{minipage}}
\caption{Type III}\label{Type III}
\end{figure}
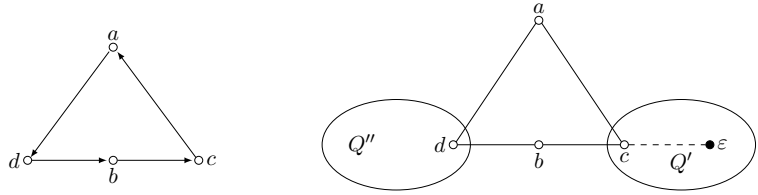
Type IV: $Q$ has a full subquiver which is an oriented $d$-cycle, where $d \geq 3$. We will call this the \textit{central cycle}. For each arrow $\alpha:a\to b$ in the central cycle, there may (and may not) be a vertex $c_{\alpha}$ which is not on the central cycle, such that there is an oriented 3-cycle traversing $a\to b \to c_{\alpha} \to a$. Moreover, this 3-cycle is a full subquiver. Such a 3-cycle will be called a \textit{spike}. There are no more arrows starting or ending in vertices on the central cycle. 

Now $Q\backslash \{ \text{vertices in the central cycle and their incident arrows}\}=Q' \cup Q'' \cup Q'''\cup \ldots$ is a disconnected union of quivers, one for each spike, which $Q'\in \mathscr{\phi}(A^{\varepsilon}_{k})$ for some integer $k$ and the others are all in $\mathcal{U}^A$ and for which the corresponding vertex $c$ is a connecting vertex, see Figure \ref{Type IV}.
\begin{figure}[H]
\resizebox{0.8\width}{0.8\height}{
\centerline{
\begin{tikzpicture}
\draw (0,0) circle (2pt);
\draw (2,0) circle (2pt);
\draw (3.5,1.4) circle (2pt);
\draw[fill] (-3,5) circle (2pt);
\draw (-1.5,1.4) circle (2pt);
\draw (3.5,3.4) circle (2pt);
\draw (-1.5,3.4) circle (2pt);
\draw (2,5) circle (2pt);
\draw (0,5) circle (2pt);
\draw (0.075,0)--(1.925,0);
\draw (2.05,0.05)--(3.45,1.35);
\draw (-0.05,0.05)--(-1.45,1.35);
\draw (3.5,1.475)--(3.5,3.325);
\draw (-1.5,1.475)--(-1.5,3.325);
\draw (-1.45,3.45)--(-0.05,4.95);
\draw[dashed] (3.45,3.45)--(2.05,4.95);
\draw[dashed] (-2.925,5)--(-1.575,5);
\draw (-1.5,5) circle (2pt);
\draw (1,6) circle (2pt);
\draw (4.5,2.4) circle (2pt);
\draw (0.075,5)--(1.925,5);
\draw (-1.5,3.475)--(-1.5,4.925);
\draw (-1.425,5)--(-0.075,5);
\draw (0.05,5.05)--(0.95,5.95);
\draw (1.95,5.05)--(1.05,5.95);
\draw (3.55,1.45)--(4.45,2.35);
\draw (3.55,3.35)--(4.45,2.45);
\draw (5.5,2.4) ellipse (1.3 cm and 0.8 cm);
\draw (1,7) ellipse (0.8 cm and 1.3 cm);
\draw (-2.5,5.2) ellipse (1.3 cm and 0.8 cm);
\node [above] at (1,7) {$Q''$};
\node [right] at (5.5,2.4) {$Q'''$};
\node [above] at (-2.5,5) {$Q'$};
\node [above] at (-1.5,5) {$c'$};
\node [right] at (-1.5,3.3) {$i_1$};
\node [right] at (-1.5,1.5) {$i_2$};
\node [above] at (0,0) {$i_3$};
\node [above] at (2,0) {$i_4$};
\node [left] at (3.5,1.5) {$i_5$};
\node [left] at (3.5,3.3) {$i_6$};
\node [below] at (2,5) {$i_{d-1}$};
\node [below] at (0,5) {$i_d$};
\node [left] at (-3,5) {$\varepsilon$};
\node [above] at (1,6) {$c''$};
\node [right] at (4.5,2.4) {$c'''$};
\end{tikzpicture}}}
\caption{Type IV}\label{Type IV}
\end{figure}

Let $\mathcal{U}(\Delta)$ be the mutation class of any $\Delta$ quiver.
\begin{lemma}\label{mutation class with frozen vertex}
The set $\mathcal{U}(\Delta)$ consists of these diagrams described as before.
\end{lemma}
\begin{proof}
The mutation class $\mathcal{U}(\Delta)$ is contained in the mutation class of the corresponding Dynkin quiver obtained by viewing the frozen vertex as a mutable vertex.

The mutation classes of Dynkin quivers of types $A_n$, $B_n$, and $D_n$ were given in \cite{BM15, FZ03, BV08, V10}. In particualr, Dynkin quivers of type $D_n$ are also obtained by using Schiffler's geometric model of cluster categories of type $D_n$ in \cite{Schi08}.

By Lemma 3.2 of \cite{V10}, we deduce that these quivers in $\mathscr{\phi}(A^{\varepsilon}_{k})$ (respectively, $\mathscr{\phi}(B^{\varepsilon}_n)$, $\mathscr{\phi}(D^{\varepsilon}_n)$) belong to the corresponding mutation class $\mathcal{U}(\Delta)$ of $\Delta$ quivers. It is not difficult to show that the set $\mathscr{\phi}(A^{\varepsilon}_{k})$ (respectively, $\mathscr{\phi}(B^{\varepsilon}_n)$) keeps invariant under the mutations. By using the similar argument as Theorem 3.1 of \cite{V10}, we show that the $\mathscr{\phi}(D^{\varepsilon}_n)$ also keeps invariant under the mutations.     
\end{proof}

\subsection{Inverse monoids determined by quivers}\label{monoids and diagrams}

Let $I \cup \{\varepsilon\}$ be the set of vertices of a quiver $Q$ with a frozen vertex $\varepsilon$. For any $i,j \in I$ and $\varepsilon$, define
\begin{align*}
& m_{ij} = \begin{cases}
2 & \text{ if $i$ and $j$ are not connected}, \\
3 & \text{ if $i$ and $j$ are connected by an edge of weight $1$}, \\
4 & \text{ if $i$ and $j$ are connected by an edge of weight $2$}, \\
6 & \text{ if $i$ and $j$ are connected by an edge of weight $3$}.
\end{cases} \\
& m_{\varepsilon j} = \begin{cases}
2 & \text{ if $\varepsilon$ and $j$ are not connected}, \\
3 & \text{ if $\varepsilon$ and $j$ are connected by an edge of weight $1$}, \\
1 & \text{ if $\varepsilon$ and $j$ are connected by an edge of weight $2$}.
\end{cases}  \\
& m_{j \varepsilon} = \begin{cases}
2 & \text{ if $\varepsilon$ and $j$ are not connected}, \\
4 & \text{ if $\varepsilon$ and $j$ are connected by an edge of weight $1$}, \\
2 & \text{ if $\varepsilon$ and $j$ are connected by an edge of weight $2$}.
\end{cases}
\end{align*}
Setting, in addition, $m_{ii}=1$ for any $i \in I\cup\{\varepsilon\}$. Then $(m_{ij})_{i,j\in I}$ is a Coxeter matrix. 
%To illustrate, generalized Coxeter matrices corresponding to a $A^{\varepsilon}_{n-1}$ quiver and a $B^{\varepsilon}_{n}$ quiver are respectively
%\begin{align*}
%\begin{pmatrix}
%1 & 3 & 2 & \cdots & \cdots & \cdots & 2 \\
%3 & 1 & 3 & 2 & \cdots & \cdots  & 2 \\
%2 & 3 & 1 & 3 & 2 & \cdots & 2 \\
%\vdots & \ddots & \ddots & \ddots & \ddots & \ddots & \vdots \\
%2 & \cdots & 2 & 3  &   1 & 3 & 2\\
%2 & \cdots & \cdots & 2  &   3 & 1 & 4 \\
%2 & \cdots & \cdots &\cdots & 2 & 3 & 1
%\end{pmatrix}_{n\times n} \textrm{ and } \quad
%\begin{pmatrix}
%1 & 4 & 2 & \cdots & \cdots & \cdots & 2 \\
%4 & 1 & 3 & 2 & \cdots & \cdots  & 2 \\
%2 & 3 & 1 & 3 & 2 & \cdots & 2 \\
%\vdots & \ddots & \ddots & \ddots & \ddots & \ddots & \vdots \\
%2 & \cdots & 2 & 3  &   1 & 3 & 2\\
%2 & \cdots & \cdots & 2  &   3 & 1 & 4 \\
%2 & \cdots & \cdots &\cdots & 2 & 3 & 1
%\end{pmatrix}_{(n+1) \times (n+1)}.
%\end{align*}

In order to define an inverse monoid $M(Q)$ associated to $Q$, we introduce some notations. Let $(i_1,i_2,\ldots,\varepsilon)$ denote the shortest path from $i_1$ to $\varepsilon$ in $Q$. We denote by $e$ and $P(s_{i_1},s_\varepsilon)$ the identity element and the element $s_{i_1}\ldots s_{\varepsilon}$ in $M(Q)$ respectively. Denote by $(aba\ldots)_m$ an alternating product of $m$ terms.

\begin{definition}\label{definition relation}
For any quiver $Q\in \mathcal{U}(\Delta)$, we define an inverse monoid $M(Q)$ 
with generators $s_i$, $i \in I\cup\{\varepsilon\}$ and relations:
\begin{itemize}
\item[(R1)] $s_{i}^2 = e$ for $i \in I$, $s_\varepsilon^2=s_\varepsilon$;
\item[(R2)] $(s_{i}s_{j})^{m_{ij}} = e$ for $i, j \in I$ and $(s_{\varepsilon}s_{j}s_{\varepsilon} \cdots)_{m_{\varepsilon j}} = (s_{j}s_{\varepsilon}s_{j} \cdots)_{m_{j\varepsilon}} = (s_{\varepsilon}s_{j}s_{\varepsilon} \cdots)_{m_{\varepsilon j}+1}$ for any $j \in I$;
\item[(R3)]
\begin{itemize}
\item[(i)] For every chordless oriented cycle $C$ in $Q$:
\begin{align*}
& i_0 \overset{w_1}{\longrightarrow}  i_1 \overset{w_2}{\longrightarrow} \cdots \to i_{d-1} \overset{w_0}{\longrightarrow} i_0,
\end{align*}
where $i_j \in I$ for $j=0,1,\ldots,d-1$, either all of the weights are 1, or $w_0=2$, we have:
\begin{align*}
& (s_{i_0} s_{i_1}\cdots s_{i_{d-2}}s_{i_{d-1}}s_{i_{d-2}}\cdots s_{i_1})^2 = e.
\end{align*}
\item[(ii)] For every chordless oriented cycle $C$ in $Q$:
\begin{align*}
& \varepsilon \longrightarrow i_1 \longrightarrow \cdots \to i_{d-1} \longrightarrow \varepsilon,
\end{align*}
where $i_j \in I$ for $j=1,\ldots,d-1$, we have:
\begin{align*}
& s_\varepsilon s_{i_1}\cdots s_{i_{d-2}}s_{i_{d-1}}s_{i_{d-2}}\cdots s_{i_1} = s_{i_1}\cdots s_{i_{d-2}}s_{i_{d-1}}s_{i_{d-2}}\cdots s_{i_1}s_\varepsilon.
\end{align*}
\item[(iii)] For every chordless oriented cycle $C$ in $Q$:
\begin{align*}
& \varepsilon \overset{w_1}{\longrightarrow}  i_1 \overset{2}{\longrightarrow} i_2 \overset{w_2}{\longrightarrow} \varepsilon,
\end{align*}
where $i_1,i_2 \in I$, if $w_1=1$ and $w_2=2$, we have $s_\varepsilon s_{i_1} s_{i_2} s_{i_1} = s_{i_1} s_{i_2} s_{i_1}s_\varepsilon$; if $w_1=2$ and $w_2=1$, we have $s_{i_1} s_{i_2} s_\varepsilon s_{i_2} =  s_{i_2} s_\varepsilon s_{i_2} s_{i_1}$.
\end{itemize}
\item[(R4)] 
\begin{itemize}
	\item[(i)] Without loss of generality, let $(0,1,\ldots,d,\varepsilon)$ be the shortest path from $0$ to $\varepsilon$ in $Q\in \mathcal{U}(B^\varepsilon_n)$, we have
	\[
	P(s_0, s_\varepsilon) = P(s_1, s_\varepsilon), \ P(s_\varepsilon,s_0) = P(s_\varepsilon,s_1).
	\]
  \item[(ii)] For any $Q\in \mathcal{U}(D^\varepsilon_n)$ drawn in Figures \ref{Type I}--\ref{Type III}, we have 
	\[
	P(s_a, s_\varepsilon)P(s_\varepsilon,s_a)= P(s_b, s_\varepsilon)P(s_\varepsilon,s_b).
	\]
	For any $Q\in \mathcal{U}(D^\varepsilon_n)$ drawn in Figure \ref{Type IV}, we have 
	\[
	s_{i_1} P(s_{c'}, s_\varepsilon)P(s_\varepsilon,s_{c'})s_{i_1}= s_{i_2}\cdots s_{i_d} P(s_{c'}, s_\varepsilon)P(s_\varepsilon,s_{c'})s_{i_d}\cdots s_{i_2}.
	\]
\end{itemize}
\end{itemize}
\end{definition}

\begin{remark}
\begin{itemize}
  \item[(1)] In (R2), if $m_{j\varepsilon}=2$ then $m_{\varepsilon j}=2$. In this case the equation $(s_{\varepsilon}s_{j}s_{\varepsilon} \ldots)_{m_{\varepsilon j}} = (s_{j}s_{\varepsilon}s_{j} \ldots)_{m_{j\varepsilon}} = (s_{\varepsilon}s_{j}s_{\varepsilon} \ldots)_{m_{\varepsilon j}+1}$ is simplified as $s_\varepsilon s_j=s_j s_\varepsilon$.
  \item[(2)] For relation (R3)~(ii), though in this paper we only use the case $d=3,4$, the defined relation for arbitrary $d$ is still meaningful, we will study it in future work.
\end{itemize}
\end{remark}

Now we are ready for our main results in this section.
\begin{theorem}\label{mutation-invariance}
Let $\Delta \in \{A^{\varepsilon}_{n-1}, B^{\varepsilon}_{n}, D^{\varepsilon}_{n}\}$ and $Q_0$ be a $\Delta$ quiver. If $Q \sim_{\text{mut}} Q_0$ then $M(Q) \cong M(Q_0)$.
\end{theorem}

We will prove Theorem \ref{mutation-invariance} in Section \ref{the proof of main theorem}. Up to the above isomorphism, we denote by $M(\Delta)$ the inverse monoid associated to any quiver appearing in $\mathcal{U}(\Delta)$.

The following example is given to explain Theorem \ref{mutation-invariance}.

\begin{example}\label{example A3}
We start with a $A^{\varepsilon}_3$ quiver $Q_0$ which is shown in Figure \ref{A3}~(a). Let $Q_1 = \mu_2(Q_0)$ be the quiver obtained from $Q_0$ by a mutation at vertex 2.
\begin{figure}[H]
\begin{align*}
(a) \ \xymatrix{
\overset{1}{\circ}  \ar[r] & \overset{2}{\circ} \ar[r] & \overset{\varepsilon}{\bullet}} \qquad \overset{\mu_{2}} {\longrightarrow} \qquad
(b) \ \xymatrix{
\overset{1}{\circ} \ar@/^20pt/@{->}[rr] & \overset{2}{\circ} \ar[l] & \overset{\varepsilon}{\bullet}\ar[l]}
\end{align*}
\caption{(a) A $A^{\varepsilon}_3$ quiver $Q_0$; (b) The quiver $Q_1 = \mu_2(Q)$.}\label{A3}
\end{figure}
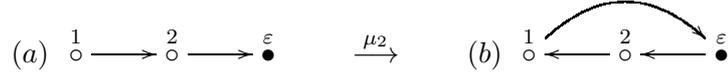
It follows from Definition \ref{definition relation} that
\begin{align*}
M(Q_0) = & \left\langle s_1, s_2, s_\varepsilon \mid s^2_1= s^2_2=e, \ s^2_\varepsilon= s_\varepsilon,\ s_1s_2s_1=s_2s_1s_2,\ s_1s_\varepsilon=s_\varepsilon s_1,\right. \\
& \left. s_2s_\varepsilon s_2 s_\varepsilon =  s_\varepsilon s_2 s_\varepsilon  s_2= s_\varepsilon s_2 s_\varepsilon \right\rangle, \\
M(Q_1) = & \left\langle t_1, t_2, t_\varepsilon \mid t^2_1= t^2_2=e, \ t^2_\varepsilon= t_\varepsilon,\ t_1t_2t_1=t_2t_1t_2,\ t_1t_\varepsilon t_1t_\varepsilon =t_\varepsilon t_1 t_\varepsilon t_1 =t_\varepsilon t_1 t_\varepsilon, \right. \\
& \left. t_2t_\varepsilon t_2 t_\varepsilon =  t_\varepsilon t_2 t_\varepsilon  t_2= t_\varepsilon t_2 t_\varepsilon,\ t_\varepsilon t_2t_1t_2 = t_2t_1t_2 t_\varepsilon \right\rangle.
\end{align*}
Then an inverse monoid isomorphism $\varphi: M(Q_0) \to M(Q_1)$ is given by
\begin{align*}
\varphi(s_i) = \begin{cases}
t_2 t_i t_2 & \text{if $i = 1$}, \\
t_i & \text{otherwise}. \\
\end{cases}
\end{align*}
\end{example}

When we say that we mutate a sequence $(n, n-1, \cdots, 2,1)$ of vertices of a quiver we mean that we first mutate the $n$-th vertex of the quiver, then we mutate the $(n-1)$-th vertex, and so on until the first vertex. The following theorem shows that Everitt and Fountain's presentations of Boolean reflection monoids can be obtained from any $\Delta$ quiver by mutations.

\begin{theorem}\label{Everitt and Fountain's presentations}
Let $\Phi=A_{n-1}, B_n$ or $D_n$. Then $M(\Phi,\mathcal{B})\cong M(\Phi^{\varepsilon})$.
\end{theorem}

\begin{proof}
In view of Theorem \ref{mutation-invariance}, we only need to prove that there exists a quiver $Q\in \mathcal{U}(\Delta)$ such that $M(Q)=M(\Phi,\mathcal{B})$. 

We mutate the sequence $(n-1, n-2, \cdots, 2,1)$ of vertices of the following $A^{\varepsilon}_{n-1}$ quiver
$$\xymatrix{
\underset{1}{\circ} \ar[r] & \underset{2}{\circ} \ar[r] & \underset{3}{\circ} \ar@{-->}[r] & \underset{n-1}{\circ} \ar[r] &  \underset{\varepsilon}{\bullet}},$$
we obtain the quiver $Q_1$:
$$\xymatrix{
\underset{\varepsilon}{\bullet} \ar[r] & \underset{1}{\circ} \ar[r] & \underset{2}{\circ} \ar@{-->}[r] & \underset{n-2}{\circ} \ar[r] &  \underset{n-1}{\circ}}.$$
Then by Definition \ref{definition relation}
\begin{align*}
M(Q_1)=  & \left\langle  s_1,s_2, \ldots, s_{n-1}, s_{\varepsilon} \mid (s_is_j)^{m_{ij}}=e\ \text{for } 1 \leq i,j\leq n-1, \ s_{\varepsilon}^2=s_{\varepsilon}, \right. \\
& \left. s_i s_{\varepsilon} = s_{\varepsilon} s_i \text{ for $i\neq 1$},\ s_{\varepsilon} s_{1} s_{\varepsilon} s_{1}= s_{1} s_{\varepsilon} s_{1} s_{\varepsilon} = s_{\varepsilon} s_{1} s_{\varepsilon} \right\rangle, \\
\end{align*}
where $m_{ij} = \begin{cases}
1 & \text{if $i=j$}, \\
3 & \text{if $|i-j|=1$}, \\
2 & \text{otherwise}.
\end{cases}$
By Lemma \ref{EF presentations}, we deduce that $M(Q_1)=M(A_{n-1},\mathcal{B})$.

%\begin{figure}[H]
%\centerline{
%\begin{tikzpicture}
%\draw[fill] (0,0) circle (2pt);
%\node [below] at (0,0) {$\varepsilon$};
%\draw (0.075,0)--(1.925,0);
%\draw (2,0) circle (2pt);
%\draw(2.075,0)--(3.925,0);
%\node [below] at (2,0) {$1$};
%\draw (4,0) circle (2pt);
%\draw[dashed] (4.075,0)--(5.925,0);
%\node [below] at (4,0) {$2$};
%\draw (6,0) circle (2pt);
%\draw (6.075,0)--(7.925,0);
%\node [below] at (6,0) {$n-2$};
%\draw (8,0) circle (2pt);
%\node [below] at (8,0) {$n-1$};
%\end{tikzpicture}}
%\end{figure}

After the sequence $(n-1, n-2, \cdots, 1,0)$ of mutations, the $B^{\varepsilon}_n$ quiver 
\[
\xymatrix{
\underset{0}{\circ} \ar[r]^-{2} & \underset{1}{\circ} \ar[r] & \underset{2}{\circ} \ar@{-->}[r] & \underset{n-1}{\circ} \ar[r] &  \underset{\varepsilon}{\bullet}}
\]
becomes the following quiver (denoted by $Q_2$)
$$\xymatrix{ & \overset{\varepsilon}{\bullet} \ar[dl]_{2} &  & & &  \\
\underset{0}{\circ} \ar@{->}[rr]_{2} &  &  \underset{1} \circ \ar[ul] \ar@{->}[r] &  \underset{2}\circ \ar@{-->}[r]  & \underset{n-2}{\circ} \ar@{->}[r] & \underset{n-1} \circ}.$$
Then by Definition \ref{definition relation}
\begin{align*}
M(Q_2) =& \left\langle s_0,s_1,\ldots,s_{n-1}, s_{\varepsilon} \mid (s_is_j)^{m_{ij}}=e\ \text{for } 0 \leq i,j\leq n-1, \right. \\
&\left. s_{\varepsilon}^2=s_{\varepsilon},\ s_0 s_{1} s_{\varepsilon} s_{1} = s_{1} s_{\varepsilon} s_{1} s_{0}, \ s_0 s_{\varepsilon} = s_{\varepsilon} s_0 = s_{\varepsilon} , \right. \\
&\left. s_i s_{\varepsilon} = s_{\varepsilon} s_i \text{ for $i\neq 1$},\ s_{1} s_{\varepsilon} s_{1} s_{\varepsilon} = s_{\varepsilon} s_{1} s_{\varepsilon} s_{1}= s_{\varepsilon} s_{1} s_{\varepsilon} \right.\rangle, \\
\end{align*}
where $m_{ij} = \begin{cases}
1 & \text{ if $i=j$}, \\
2 & \text{ if $i$ and $j$ are not connected}, \\
3 & \text{ if $i$ and $j$ are connected by an edge with weight $1$}, \\
4 & \text{ if $i$ and $j$ are connected by an edge with weight $2$}.
\end{cases}$
By Lemma \ref{EF presentations}, we have $M(Q_2)=M(B_{n},\mathcal{B})$.

%& \overset{\text{Lemma} \ref{EF presentations}}{=} M(B_{n},\mathcal{B}).

%\begin{figure}[H]
%\resizebox{1.0\width}{1.0\height}{
%\centerline{
%\begin{tikzpicture}
%\draw[fill] (3,1.5) circle (2pt);
%\node [above] at (3,1.6) {$\varepsilon$};
%%\draw (0.15,0)--(1.85,0);
%\node [below] at (3,0) {$2$};
%\draw (2,0) circle (2pt);
%\draw (2.075,0)--(3.925,0);
%\node [below] at (2,0) {$0$};
%\draw (4,0) circle (2pt);
%\draw (4.075,0)--(5.925,0);
%\node [below] at (4,0) {$1$};
%\draw (6,0) circle (2pt);
%\draw[dashed] (6.075,0)--(7.925,0);
%\node [below] at (6,0) {$2$};
%\draw (8,0) circle (2pt);
%\draw (8.075,0)--(9.925,0);
%\node [below] at (8,0) {$n-2$};
%\draw(10,0) circle (2pt);
%\node [below] at (10,0) {$n-1$};
%\draw (2.95,1.45)--(2.05,0.05);
%\draw (3.05,1.45)--(3.95,0.05);
%\node [above] at (2.3,0.5) {$2$};
%\end{tikzpicture}}}
%\end{figure}

After the sequence $(n-1, n-2, \cdots, 2,0)$ of mutations, the $D^{\varepsilon}_n$ quiver 
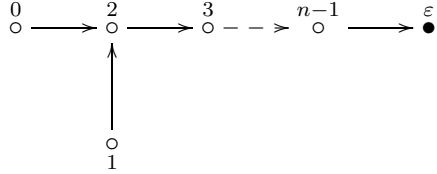
\[
\xymatrix{\overset{0}{\circ} \ar[r] & \overset{2}{\circ} \ar[r]  &  \overset{3}{\circ} \ar@{-->}[r]  &  \overset{n-1}{\circ} \ar[r] & \overset{\varepsilon}{\bullet}\\
& \underset{1}{\circ} \ar[u] & &  & }
\]
becomes the following quiver (denoted by $Q_3$)
$$
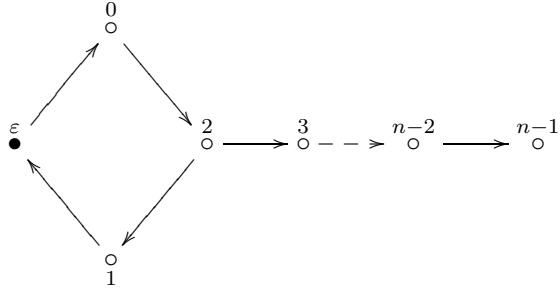
\xymatrix{& \overset{0} \circ \ar[dr] & &  & & \\
\overset{\varepsilon}{\bullet} \ar[ur] & &  \overset{2}{\circ}\ar[dl]\ar[r] & \overset{3}{\circ} \ar@{-->}[r] & \overset{n-2}{\circ} \ar[r] & \overset{n-1}
{\circ} &  \\
& \underset{1}{\circ} \ar[ul]& & &  &}.$$
Then by Definition \ref{definition relation}
\begin{align*}
M(Q_3)= & \left\langle s_0,s_1,\ldots,s_{n-1},s_{\varepsilon} \mid (s_is_j)^{m_{ij}}=e\ \text{for } 0 \leq i,j\leq n-1,\ s_{\varepsilon}^2=s_{\varepsilon}, \right. \\
&\left. s_i s_{\varepsilon} = s_{\varepsilon} s_i \text{ for $i> 1$},\ s_{\varepsilon} s_{j} s_{\varepsilon} s_{j}= s_{j} s_{\varepsilon} s_{j} s_{\varepsilon} = s_{\varepsilon} s_{j} s_{\varepsilon} \text{ for $j=0,1$}, \right. \\
&\left. s_0 s_{\varepsilon} s_{0}  = s_1 s_{\varepsilon} s_1,\ s_0 s_2 s_1 s_{\varepsilon} s_1 s_2 = s_2 s_1 s_{\varepsilon} s_1 s_2 s_0 \right\rangle,
\end{align*}
where
\begin{align*}
m_{ij} = \begin{cases}
1 & \text{ if $i=j$}, \\
2 & \text{ if $i$ and $j$ are not connected}, \\
3 & \text{ if $i$ and $j$ are connected by an edge with weight $1$}.
\end{cases}
\end{align*}

We claim that $M(Q_3)=M(D_{n},\mathcal{B})$, which follows from Lemma \ref{a chordless cycle lemma in M}~(3) and Lemma \ref{EF presentations}.
\end{proof}

For any quiver $Q\in\mathcal{U}(\Delta)$, by Theorem \ref{mutation-invariance} and Theorem \ref{Everitt and Fountain's presentations}, $M(Q)$ gives a presentation of the corresponding Boolean reflection monoid. In particular, our presentations recover the presentation of the symmetric inverse semigroup $\mathscr{I}_{n}$ defined in \cite{E11, P61}.

\subsection{Inner by diagram automorphisms of Boolean reflection monoids}
We first consider inner automorphisms of the Boolean reflection monoid $M(\Phi,\mathcal{B})$. 

An automorphism $\alpha$ of $M(\Phi,\mathcal{B})$ is called an \textit{inner automorphism} if there exists a uniquely determined element $g\in W(\Phi)$ of the Weyl group $W(\Phi)$ such that $\alpha(t)=gtg^{-1}$ for all $t\in M(\Phi,\mathcal{B})$. Let $\text{Aut}(G)$ (respectively, $\text{Inn}(G)$) be the automorphism  (respectively, inner automorphism) group of $G$ and $Z(G)$ be the center of $G$ for any semigroup $G$. 

It was showed in \cite{L53,ST97} that $\text{Aut}(M(A_{n-1},\mathcal{B}))=\text{Inn}(M(A_{n-1},\mathcal{B}))$. In other words, $\text{Aut}(M(A_{n-1},\mathcal{B}))\cong W(A_{n-1})/ Z(W(A_{n-1}))$ for $n\geq 3$ and $\text{Aut}(M(A_1,\mathcal{B}))\cong W(A_1)$.

As a generalization of the above result, we have the following theorem.

\begin{theorem}\label{inner automorphisms}
$\text{Inn}(M(\Phi,\mathcal{B}))\cong W(\Phi)/Z(W(\Phi))$ for $\Phi=A_{n-1} (\geq 3),\; B_n (n\geq 2)$,\; $D_n (n \geq 4)$.
\end{theorem}
\begin{proof}
Let $\alpha\in \text{Inn}(M(\Phi,\mathcal{B}))$. Since $W(\Phi) \subseteq M(\Phi,\mathcal{B})$ is the unique unit group of $M(\Phi,\mathcal{B})$, we have the fact that $\alpha|_{W(\Phi)}$ is an inner automorphism of $W(\Phi)$.

We will prove the parts of $\Phi=B_n (n\geq 2), \;D_n (\geq 4)$. Let $\Phi^\varepsilon$ be one of $B^\varepsilon_{n}$ and $D^\varepsilon_{n}$ shown in Table \ref{initial ABD}. Suppose that $\Lambda=\{s_0, s_1, \ldots, s_{n-1}, s_\varepsilon\}$ is a set of generators of $M(\Phi,\mathcal{B})$. For any element $g\in W(\Phi)$, we claim that the set 
\[
g\Lambda g^{-1}=\{gs_0g^{-1}, gs_1g^{-1}, \ldots, gs_{n-1}g^{-1}, gs_\varepsilon g^{-1}\}\]
is still a set of generators of $M(\Phi,\mathcal{B})$. Obviously, $(gs_ig^{-1})^2=e$, $(gs_\varepsilon g^{-1})^2=gs_\varepsilon g^{-1}$, and $g\Lambda g^{-1}$ satisfies (R2)--(R4) in Definition \ref{definition relation}.

Finally, we show that $g_1s_\varepsilon g^{-1}_1=g_2s_\varepsilon g^{-1}_2$ for any $g_1,g_2 \in Z(W(\Phi))$. It suffices to prove that $s_\varepsilon= w_0s_\varepsilon w^{-1}_0$, where $w_0$ is the longest element in $W(\Phi)$ and $w_0$ is an involution. By Section 1.2 of \cite{Fran01}, we have $w_0= w_n w_{n-1} \cdots w_1$, where $w_i = s_{i-1} \cdots s_1 s_0 s_1\cdots s_{i-1}$ for $i\geq 1$ in type $\Phi=B_n$, and $w_i=s_{i-1}\cdots s_3 s_2 s_1 s_0 s_2 s_3 \cdots s_{i-1}$ for $i\geq 3$ and $w_1=s_0$, $w_2=s_1$ in type $\Phi=D_n$. Then by (R1), (R2), and (R4) of Definition \ref{definition relation},
\begin{align*}
& w_0 s_\varepsilon w^{-1}_0 =w_n w_{n-1} \cdots w_1  s_\varepsilon w_1 \cdots w_{n-1} w_n = w_n s_\varepsilon w_n \\
& = \begin{cases}
s_{n-1} \cdots s_1 (s_0 s_1\cdots s_{n-1}  s_\varepsilon s_{n-1} \cdots s_1 s_0) s_1\cdots s_{n-1}= s_\varepsilon & \text{ in type $B_n$},\\
s_{n-1} \cdots s_3 s_2 s_1 (s_0 s_2 s_3 \cdots s_{n-1} s_\varepsilon s_{n-1} \cdots s_3 s_2 s_0) s_1 s_2 s_3 \cdots s_{n-1}= s_\varepsilon & \text{ in type $D_n$}.
\end{cases}
\end{align*}

The theorem is proved.
\end{proof}

Let $\Delta$ be one of $A^\varepsilon_{n-1}$, $B^\varepsilon_{n}$, and $D^\varepsilon_{n}$ shown in Table \ref{initial ABD}. Let $Q$ be a $\Delta$ quiver. Let $I \cup \{\varepsilon\}$ be the set of vertices of $Q$ and $Q'=\mu_{k}(Q)$ the quiver obtained by a mutation of $Q$ at a mutable vertex $k$. Following Barot and Marsh's work \cite{BM15}, one can define variables $t_i$ for $i \in I$, and $t_{\varepsilon}$ in $M(Q')$ as follows:
\begin{align}\label{formular of mutations}
\begin{split}
&t_i=
\begin{cases}
s_{k}s_{i}s_{k} & \text{if there is an arrow $i \to k$ in $Q$ (possibly weighted)}, \\
s_i & \text{otherwise},
\end{cases}\\
&t_{\varepsilon}=
\begin{cases}
s_{k}s_{\varepsilon} s_{k} & \text{if there is an arrow $\varepsilon \to k$ in $Q$ (possibly weighted)}, \\
s_{\varepsilon} & \text{otherwise}.
\end{cases}
\end{split}
\end{align}
From Lemma~\ref{the form of reflections} and Equation~(\ref{formular of mutations}), it follows that new elements $t_i$, $i\in I$, appearing in the procedure of mutations of quivers, must be some reflections in Weyl groups.

We introduce the definition of inner by diagram automorphisms of $M(\Phi,\mathcal{B})$.
\begin{definition}\label{inner by diagram automorphism of semigroups}
An inner by diagram automorphism of $M(\Phi,\mathcal{B})$ is an automorphism generated by some inner automorphisms and diagram automorphisms in $\text{Aut}(M(\Phi,\mathcal{B}))$.
\end{definition}

For simplicity, let $W(Q\backslash\{\varepsilon\})$ be the unit group of $M(Q)$. In the following theorem, we show that inner by diagram automorphisms of $M(\Phi,\mathcal{B})$ can be constructed by a sequence of mutations preserving the same underlying diagrams.

\begin{theorem}\label{mutate inner automorphisms}
Let $Q$ be a $\Delta$ quiver and $M(Q)$ the corresponding Boolean reflection monoid with a set $S$ of generators. Then $\alpha$ is an inner by diagram automorphism of $M(Q)$ if and only if there exists a sequence of mutations preserving the underlying diagram $\Delta$ such that $\alpha(S)$ can be obtained from $Q$ by mutations. In particular, all reflections in $W(Q\backslash\{\varepsilon\})$ and $g s_\varepsilon g^{-1}$ for $g\in W(Q\backslash\{\varepsilon\})$ can be obtained from $Q$ by mutations.
\end{theorem}

\begin{proof}
Suppose that $S=\{s_1, s_2, \ldots, s_{n-1}, s_\varepsilon\}$ for $\Delta=A^\varepsilon_{n-1}$ or $S=\{s_0, s_1, \ldots, s_{n-1}, s_\varepsilon\}$ for $\Delta=B^\varepsilon_{n},\; D^\varepsilon_{n}$.

All automorphisms of $M(A_{n-1},\mathcal{B})$ are inner, see Theorem \ref{inner automorphisms}. A sequence $\mu$ of mutations preserving the underlying diagram of $Q$ induces to an inner automorphism of $M(A_{n-1},\mathcal{B})$. 

The remainder proof of the sufficiency is to prove types $B_n$ and $D_n$. Since all automorphisms of irreducible Weyl groups that preserve reflections are inner by diagram automorphisms, we assume without loss of generality that $M(\mu(Q))$ is generated by $\{t_0, t_1, \ldots, t_{n-1}, t_\varepsilon\}$, where $t_i = gs_ig^{-1}$ for $0 \leq i \leq n-1$, $g \in W(B_n)$ or $W(D_n)$. We claim that $t_\varepsilon = gs_\varepsilon g^{-1}$. If $t_\varepsilon = gs_\varepsilon g^{-1}$, then $\{t_0, t_1, \ldots, t_{n-1}, t_\varepsilon\}$ is a set of generators of $M(\mu(Q))$ and $\mu(Q)$ preserves the underlying diagram $\Delta$. 

On the one hand, since $t_\varepsilon t_i = t_i t_\varepsilon$ for $0 \leq i \leq n-2$, we have $t_\varepsilon \in Z(W(B_{n-1}))$ or $Z(W(D_{n-1}))$, where $\{t_\varepsilon t_0, t_\varepsilon t_1, \ldots, t_\varepsilon t_{n-2}\}$ is a set of generators of $W(B_{n-1})$ or $W(D_{n-1})$. On the other hand, the variable $t_\varepsilon$ must be of the form $g's_\varepsilon {g'}^{-1}$ for some $g' \in W(B_n)$ or $W(D_{n})$). So $t_\varepsilon$ is not the longest element $w_0$ in $W(B_{n-1})$ or $W(D_{n-1})$ and hence $t_\varepsilon$ must be the unique identity element in $W(B_{n-1})$ or $W(D_{n-1})$. Therefore by the uniqueness $t_\varepsilon=gs_\varepsilon g^{-1}$.

Conversely, for each inner automorphism $\alpha$ of $M(Q)$, by Theorem \ref{inner automorphisms}, there exists an element $g\in W(Q\backslash \{\varepsilon\})$ such that $\alpha(t)=gtg^{-1}$ for all $t\in M(Q)$. The remainder proof of the necessity is similar to the proof of the necessity of Theorem \ref{mutate inner automorphisms of weyl groups}.

Each reflection in $W(Q\backslash \{\varepsilon\})$ is of the form $gs_ig^{-1}$, where $g=s_{i_1}s_{i_2}\cdots s_{i_k} \in W(Q\backslash \{\varepsilon\})$ is a reduced expression for $g$. By the same arguments as Theorem \ref{mutate inner automorphisms of weyl groups}, we mutate the sequence $i_1, i_1, i_2,i_2, \ldots, i_k, i_k$ starting from $Q$, we obtain $gs_ig^{-1}$ and $g s_\varepsilon g^{-1}$.
\end{proof}

\subsection{Cellularity of semigroup algebras of Boolean reflection monoids}
In this section, we show that semigroup algebras of Boolean reflection monoids are cellular algebras. We use the presentations we obtained to interpret cellular bases of such cellular algebras.

Let $R$ be a commutative ring with identity. Recall that a semigroup S is said to be cellular if its semigroup algebra $R[S]$ is a cellular algebra.

\begin{proposition}\label{cellularity of Boolean reflection monoid}
The Boolean reflection monoid $M(\Phi,\mathcal{B})$ for $\Phi= A_{n-1}$, $B_{n}$, or $D_{n}$ is a cellular semigroup.
\end{proposition}
\begin{proof}
All maximal subgroups of the Boolean reflection monoid $M(\Phi,\mathcal{B})$ are finite reflection groups. It has been shown in \cite{G07} that any finite reflection group $W(\Phi)$ is cellular with respect to which the anti-involution is inversion. Therefore, for each $\mathcal{D}$-class $D$ of $M(\Phi,\mathcal{B})$, the subgroup $H_D\subseteq M(\Phi,\mathcal{B})$ is cellular with cell datum $(\Lambda_D,M_D,C_D,i_D)$, where $i_D: H_D \to H_D$ is given by $i_D(g)=g^{-1}$ for any $g\in H_D$, which satisfies East's first assumption, see Theorem 15 in \cite{E06} or Theorem \ref{East's results}.

We define a map
\begin{align*}
i: R[M(\Phi,\mathcal{B})] & \to R[M(\Phi,\mathcal{B})] \\
\sum_{j}r_jg_j & \mapsto \sum_{j}r_jg^{-1}_j,
\end{align*}
where $r_j\in R$, $g_j\in M(\Phi,\mathcal{B})$. The map $i$ is an $R$-linear anti-homomorphism and $i([e,f,g]_D)=[e,f,g]^{-1}_D=[f,e,g^{-1}]_D=[f,e,i_D(g)]_D$ for any $g \in H_D$, $e \ \mathcal{D}\ f$ in $M(\Phi,\mathcal{B})$.

From Theorem 19 in \cite{E06} or Theorem \ref{East's results}, it follows that the Boolean reflection monoid $M(\Phi,\mathcal{B})$ is a cellular semigroup, as required.
\end{proof}
\begin{remark}
A finite inverse semigroup whose maximal subgroups are direct products of symmetric groups has been considered by East, see Theorem 22 of \cite{E06}. Maximal subgroups of $M(A_{n-1},\mathcal{B})$ are isomorphic to symmetric groups. Maximal subgroups of $M(B_n,\mathcal{B})$ (respectively, $M(D_n,\mathcal{B})$) are isomorphic to finite reflection groups of type $B_r$ (respectively, finite reflection groups of type $D_r$), where $r \leq n$.
\end{remark}

%Let $G$ be a group (respectively, semigroup) and $R[G]$ the corresponding group algebra (respectively, semigroup algebra). If $\varphi$ is an isomorphism between $G_1$ and $G_2$, then $\varphi$ can be  extended an $R$-linear isomorphism between $R[G_1]$ and $R[G_2]$, i.e. the following diagram commuts
%\begin{align*}
%\xymatrix{
%  G_1 \ar[d] \ar[r]^{\varphi} &  G_2 \ar[d] \\
%  R[G_1] \ar[r]^{\overline{\varphi}} & R[G_2].}
%\end{align*}
%In order to give new cellular bases of $R[G]$, by Corollary \ref{new cellular bases} we only need to construct an $R$-linear automorphism $\overline{\varphi}$ of $R[G]$.

%Theorem \ref{Mash's result}, Theorem \ref{mutation-invariance} and Theorem \ref{Everitt and Fountain's presentations} imply that presentations of finite irreducible crystallographic reflection groups and of Boolean reflection monoids are compatible with quiver mutations, respectively.

Let $\Delta$ be one of $A^\varepsilon_{n-1}$, $B^\varepsilon_{n}$, and $D^\varepsilon_{n}$ shown in Table \ref{initial ABD}. For two quivers with the same underlying diagrams appearing in $\mathcal{U}(\Delta)$, we construct inner by diagram automorphisms of Boolean reflection monoids, see Theorem \ref{mutate inner automorphisms}, and then we extend it to an $R$-linear automorphism of semigroup algebras of Boolean reflection monoids. Applying Corollary \ref{new cellular bases}, we interpret cellular bases of semigroup algebras of Boolean reflection monoids in terms of the presentations and inner by diagram automorphisms we obtained.

\subsection{An example}
In this section, we denote by $\mathscr{I}_n$ the symmetric inverse semigroup on $[n]=\{1,2,\ldots,n\}$. Let $w$ be a partial permutation on a set $A \subseteq [n]$ and denote the image of $i\in A$ under the map $w$ by $w_i$ and the image of $i\not\in A$ under the map $w$ by $w_i = \emptyset$. Then $w$ is denoted by $\begin{pmatrix}
    1 & 2 &  \cdots & n-1 & n \\
    w_1 & w_2 & \cdots &  w_{n-1}  & w_n
  \end{pmatrix}$, see \cite{L96}. A partial permutation $w$ is called an element of rank $i$ if the number of non-empty entries $w_j$ for $1\leq j\leq n$ is $i$.

%For example, $\begin{pmatrix}
%    1 & 2 & 3 \\
%    3 & - & 2
%  \end{pmatrix}$ is the partial permutation with domain $\{1, 3\}$ and range $\{2, 3\}$ under which $1\to3$, $3\to2$, and its rank is 2.

\begin{example}
Let $Q_0$ be the quiver in Example \ref{example A3} and by the results of preceding sections, $M(Q_0) \cong \mathscr{I}_3$ a Boolean reflection monoid. We have $M(Q_0)/\mathcal{D} = \{D_0 < D_1 < D_2 < D_3\}$, where each $D_i$ is the set of all elements of $M(Q_0)$ of rank $i$, and idempotents in each $D_i$ are all partial identity permutations of rank $i$.

Let $A$ be a subset of $\{1,2,3\}$. As shown in Example 23 of \cite{E06}, the $\mathcal{H}$-class containing the idempotent $\textrm{id}_A$ is the subgroup $\{x\in \mathscr{I}_3 \;|\; \textrm{im}(x) = \textrm{dom}(x) = A\} \cong S_{|A|}$. It is well known that the group algebra $R[S_n]$ has cellular bases with respect to which the anti-involution is inversion. Indeed the Khazdan-Luzstig bases and the Murphy basis both have this property (see, Example (1.2) of \cite{GL96}, Example (2.2) of \cite{M99} or Section 4 of \cite{M95}).

Let $s_1 = \begin{pmatrix}
    1 & 2 & 3 \\
    2 & 1 & 3
  \end{pmatrix}$, $s_2 =\begin{pmatrix}
    1 & 2 & 3 \\
    1 & 3 & 2
  \end{pmatrix}$, and $s_\varepsilon=\begin{pmatrix}
    1 & 2 & 3 \\
    1 & 2 & \emptyset
  \end{pmatrix}$. By mutating the quiver $Q_0$, we obtain the following isomorphic quivers (Theorems \ref{inner automorphisms} and \ref{mutate inner automorphisms}).
\begin{figure}[H]
\begin{align*}
& (a) \ \xymatrix{
\overset{s_1}{\circ}  \ar[r] & \overset{s_2}{\circ} \ar[r] & \overset{s_\varepsilon}{\bullet}} \\
& (b) \ \xymatrix{
\overset{s_1}{\circ}  \ar[r] & \overset{s_1s_2s_1}{\circ} \ar[r] & \overset{s_\varepsilon}{\bullet}} \\
& (c) \ \xymatrix{
\overset{s_2s_1s_2}{\circ}  \ar[r] & \overset{s_2}{\circ} \ar[r] & \overset{s_2s_\varepsilon s_2}{\bullet}} \\
& (d) \ \xymatrix{
\overset{s_2}{\circ}  \ar[r] & \overset{s_1s_2s_1}{\circ} \ar[r] & \overset{s_1s_2s_\varepsilon s_2s_1}{\bullet}} \\
& (e) \ \xymatrix{
\overset{s_2s_1s_2}{\circ}  \ar[r] & \overset{s_1}{\circ} \ar[r] & \overset{s_2s_\varepsilon s_2}{\bullet}} \\
& (f) \ \xymatrix{
\overset{s_2}{\circ}  \ar[r] & \overset{s_1}{\circ} \ar[r] & \overset{s_1s_2s_\varepsilon s_2s_1}{\bullet}}
\end{align*}
%\caption{Isomorphic quivers}\label{isomorphic quivers}
\end{figure}
From Theorem~\ref{mutation-invariance} and Theorem \ref{Everitt and Fountain's presentations}, it follows that the inverse monoids associated to quivers (a)--(f) are isomorphic to the symmetric inverse semigroup $\mathscr{I}_3$ respectively. The presentation of $\mathscr{I}_3$ determined by the quiver (a) admits an initial cellular bases by Theorem 19 of \cite{E06} or Theorem \ref{East's results}. By using these presentations corresponding to quivers (a)--(f), we construct an $R$-linear automorphism of $R[\mathscr{I}_3]$ and then applying Corollary \ref{new cellular bases}, we give an alternative interpretation of cellular bases of $R[\mathscr{I}_3]$.
\end{example}

\section{Mutations of quivers of finite type}\label{diagrams of finite types}
Throughout this section, let as before $\Delta$ be one of $A^\varepsilon_{n-1}$, $B^\varepsilon_{n}$, and $D^\varepsilon_{n}$ in Table \ref{initial ABD}. We consider the way of mutations of $\Delta$ quivers and the oriented cycles appearing in them, refer to \cite{BM15, FZ03}. We first recall:

\begin{proposition}[{\cite[Proposition 9.7]{FZ03}}]\label{FZ oriented cycle}
Let $Q$ be a quiver of finite type. Then a chordless cycle in the underlying diagram of $Q$ is always cyclically oriented in $Q$. In addition, the chordless cycle must be one of those shown in Figure \ref{the oriented cycle in finite type}.
\end{proposition}
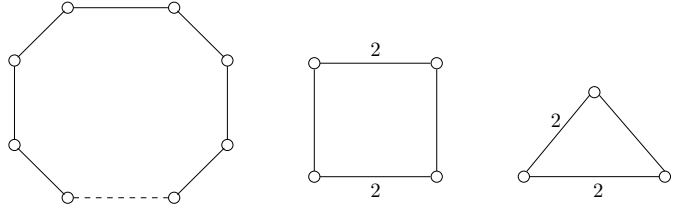
\begin{figure}[H]
\resizebox{.7\width}{.7\height}{
\begin{tikzpicture}
\draw (0.1075,1)--(1.8825,1);
\draw (0,1) circle (3pt);
\draw (2,1) circle (3pt);
\draw (-1,0) circle (3pt);
\draw (-0.94,0.06)--(-0.06,0.94);
\draw (3,0) circle (3pt);
\draw (2.94,0.06)--(2.06,0.94);
\draw (-1,-1.6) circle (3pt);
\draw (-1,-0.1075)--(-1,-1.4925);
\draw (3,-1.6) circle (3pt);
\draw (3,-0.1075)--(3,-1.4925);
\draw (0,-2.6) circle (3pt);
\draw (-0.94,-1.66)--(-0.06,-2.54);
\draw (2,-2.6) circle (3pt);
\draw[dashed] (0.1075,-2.6)--(1.8825,-2.6);
\draw (2.94,-1.66)--(2.06,-2.54);
\end{tikzpicture}} \ \ \ \ \ \
\resizebox{.7\width}{.7\height}{
\begin{tikzpicture}
\draw (0.1075,0)--(2.1825,0);
\node [below] at (1.15,0) {2};
\draw (0.1075,2.15)--(2.1825,2.15);
\node [above] at (1.15,2.15) {2};
\draw (0,0.1075)--(0,2.0325);
\draw (2.3,0.1075)--(2.3,2.0325);
\draw (0,0) circle (3pt);
\draw (0,2.15) circle (3pt);
\draw (2.3,0) circle (3pt);
\draw (2.3,2.15) circle (3pt);
\end{tikzpicture}} \ \  \ \  \ \
\resizebox{.7\width}{.7\height}{
\begin{tikzpicture}
\draw (0.11,1)--(2.54,1);
\node [below] at (1.4,1) {2};
\draw (0.055,1.1)--(1.243,2.55);
\node [above] at (0.6,1.8) {2};
\draw (1.42,2.54)--(2.65,1.1);
\draw (0,1) circle (3pt);
\draw (1.325,2.6) circle (3pt);
\draw (2.65,1) circle (3pt);
\end{tikzpicture}}
\caption{The chordless cycles in $Q$.}\label{the oriented cycle in finite type}
\end{figure}

We extend Corollary 2.3 in \cite{BM15} to the case of $\Delta$ quivers.

\begin{lemma} \label{three vertices subdiagram mutation corollary}
Let $Q$ be a $\Delta$ quiver and $k$ a mutable vertex of $Q$. Suppose that $k$ has two neighbouring vertices. Then the effect of the mutation of $Q$ at $k$ on the induced subdiagram must be as in Figure \ref{three vertices subdiagram mutation 1} (starting on one side or the other), up to switching the two neighbouring vertices.
\end{lemma}
\begin{proof}
This follows from Proposition \ref{FZ oriented cycle} and Corollary 2.3 in \cite{BM15} by restricting quivers to Dynkin quivers of types $A_n$, $B_{n+1}$, and $D_{n+1}$ with a frozen vertex $\varepsilon$.
\end{proof}

\begin{figure}[H]
$\begin{array}{lcl}
& \resizebox{.8\width}{.8\height}{
(a) \xymatrix{
& \overset{k}{\circ} &  \\
\underset{i}{\circ}  \ar[ur] &  &  \underset{j}{\circ} \ar[ul]} $\overset{\mu_{k}}{\longleftrightarrow}$
\xymatrix{
& \overset{k}{\circ} \ar[dr] \ar[dl] &  \\
\underset{i}{\circ}&  &  \underset{j}{\circ}}}  & \resizebox{.8\width}{.8\height}{
(b)\xymatrix{
& \overset{k}{\circ} \ar[dr] &  \\
\underset{i}{\circ}  \ar[ur] &  &  \underset{j}{\circ}}  $\overset{\mu_{k}}{\longleftrightarrow}$
\xymatrix{
& \overset{k}{\circ} \ar[dl] &  \\
\underset{i}{\circ} \ar[rr] & & \underset{j}{\circ} \ar[ul]}} \\
& \resizebox{.8\width}{.8\height}{
(c)\xymatrix{
& \overset{k}{\circ} &  \\
\underset{i}{\circ}  \ar[ur] &  &  \underset{j}{\circ} \ar@{->}[ul]_{2}} $\overset{\mu_{k}}{\longleftrightarrow}$
\xymatrix{
& \overset{k}{\circ} \ar@{->}[dr]^{2} \ar[dl] &  \\
\underset{i}{\circ}&  &  \underset{j}{\circ}}}
& \resizebox{.8\width}{.8\height}{
(d)\xymatrix{
& \overset{k}{\circ} \ar[dr] &  \\
\underset{i}{\circ}  \ar@{->}[ur]^{2} &  &  \underset{j}{\circ}}  $\overset{\mu_{k}}{\longleftrightarrow}$
\xymatrix{
& \overset{k}{\circ} \ar@{->}[dl]_{2} &  \\
\underset{i}{\circ} \ar@{->}[rr]_{2} & & \underset{j}{\circ} \ar[ul] }} \\
& \resizebox{.8\width}{.8\height}{
(e)\xymatrix{
& \overset{k}{\circ} \ar@{->}[dr]^{2} &  \\
\underset{i}{\circ}  \ar[ur] &  &  \underset{j}{\circ}} $\overset{\mu_{k}}{\longleftrightarrow}$
\xymatrix{
& \overset{k}{\circ} \ar[dl] &  \\
\underset{i}{\circ} \ar@{->}[rr]_{2} &  &  \underset{j}{\circ} \ar@{->}[ul]_{2} }} & \resizebox{.8\width}{.8\height}{
(f)\xymatrix{
& \overset{k}{\circ} \ar@{->}[dr]^{2}&  \\
\underset{i}{\circ}  \ar@{->}[ur]^{2} &  &  \underset{j}{\circ} \ar[ll] } $\overset{\mu_{k}}{\longleftrightarrow}$
\xymatrix{
& \overset{k}{\circ} \ar@{->}[dl]_{2} &  \\
\underset{i}{\circ} \ar[rr] & & \underset{j}{\circ} \ar@{->}[ul]_{2}}}\\
& \resizebox{.8\width}{.8\height}{
(a') \xymatrix{
& \overset{k}{\circ} &  \\
\underset{i}{\circ}  \ar[ur] &  &  \underset{\varepsilon}{\bullet} \ar[ul] } $\overset{\mu_{k}}{\longleftrightarrow}$
\xymatrix{
& \overset{k}{\circ} \ar[dr] \ar[dl] &  \\
\underset{i}{\circ}&  &  \underset{\varepsilon}{\bullet}}} & \resizebox{.8\width}{.8\height}{
(b')\xymatrix{
& \overset{k}{\circ} \ar[dr] &  \\
\underset{i}{\circ}  \ar[ur] &  &  \underset{\varepsilon}{\bullet}}  $\overset{\mu_{k}}{\longleftrightarrow}$
\xymatrix{
& \overset{k}{\circ} \ar[dl] &  \\
\underset{i}{\circ} \ar[rr] & & \underset{\varepsilon}{\bullet} \ar[ul]}}  \\
& \resizebox{.8\width}{.8\height}{
(c')\xymatrix{
& \overset{k}{\circ} \ar[dl] &  \\
\underset{i}{\circ} &  &  \underset{\varepsilon}{\bullet} \ar[ul] }  $\overset{\mu_{k}}{\longleftrightarrow}$
\xymatrix{
& \overset{k}{\circ} \ar[dr] &  \\
\underset{i}{\circ} \ar[ur] & & \underset{\varepsilon}{\bullet} \ar[ll]}} &
\resizebox{.8\width}{.8\height}{
(d')\xymatrix{
& \overset{k}{\circ} &  \\
\underset{i}{\circ} \ar@{->}[ur]^{2} &  &  \underset{\varepsilon}{\bullet} \ar@{->}[ul]}  $\overset{\mu_{k}}{\longleftrightarrow}$
\xymatrix{
& \overset{k}{\circ} \ar@{->}[dl]_{2} \ar@{->}[dr] &  \\
\underset{i}{\circ}  &  &  \underset{\varepsilon}{\bullet}}} \\
& \resizebox{.8\width}{.8\height}{
(e')\xymatrix{
& \overset{k}{\circ} \ar[dr] &  \\
\underset{i}{\circ}  \ar@{->}[ur]^{2} &  &  \underset{\varepsilon}{\bullet}}  $\overset{\mu_{k}}{\longleftrightarrow}$
\xymatrix{
& \overset{k}{\circ} \ar@{->}[dl]_{2} &  \\
\underset{i}{\circ} \ar@{->}[rr]_{2} & & \underset{\varepsilon}{\bullet} \ar[ul]}} & \resizebox{.8\width}{.8\height}{
(f')\xymatrix{
& \overset{k}{\circ} \ar@{->}[dl]_{2} &  \\
\underset{i}{\circ}  &  &  \underset{\varepsilon}{\bullet}\ar@{->}[ul]}  $\overset{\mu_{k}}{\longleftrightarrow}$
\xymatrix{
& \overset{k}{\circ} \ar@{->}[dr] &  \\
\underset{i}{\circ} \ar@{->}[ur]^{2} & & \underset{\varepsilon}{\bullet} \ar@{->}[ll]^{2}}} \\
& \resizebox{.8\width}{.8\height}{
(g')\xymatrix{
& \overset{k}{\circ} \ar@{->}[dl]_{2} &  \\
\underset{i}{\circ} \ar@{->}[rr] & & \underset{\varepsilon}{\bullet} \ar[ul]_{2}} $\overset{\mu_{k}}{\longleftrightarrow}$
\xymatrix{
& \overset{k}{\circ} \ar@{->}[dr]^{2} &  \\
\underset{i}{\circ} \ar@{->}[ur]^{2} & & \underset{\varepsilon}{\bullet} \ar@{->}[ll]}}\\
\end{array}$
\caption{Local pictures of mutations of $Q$.}\label{three vertices subdiagram mutation 1}
\end{figure}
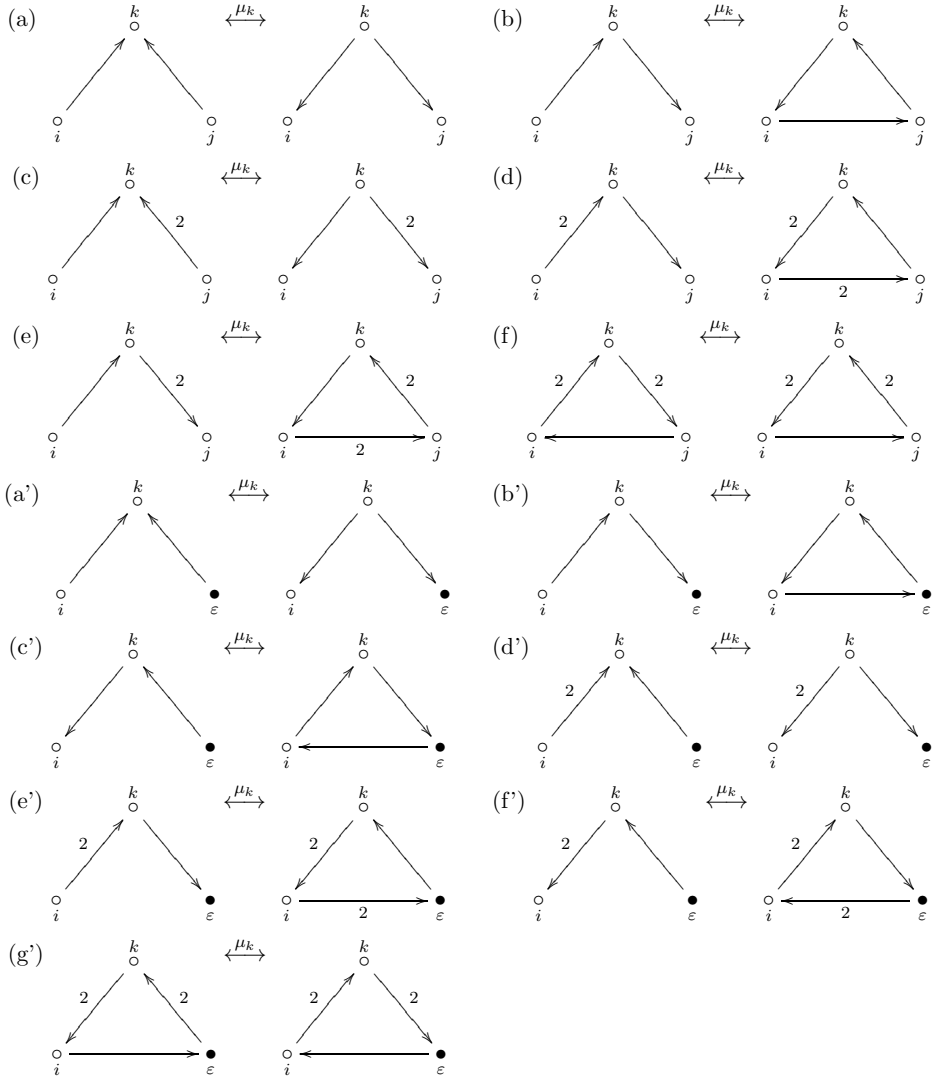

In a diagram, a vertex is said to be \textit{connected} to another if there is an edge between them. Let $Q$ be a quiver that is mutation equivalent to a Dynkin quiver. In Lemma 2.4 of \cite{BM15}, Barot and Marsh described the way that vertices in $Q$ can be connected to a chordless cycle: A vertex is connected to at most two vertices of a chordless cycle, and if it is connected to two vertices, then the two vertices must be adjacent in the cycle.

The following lemma is a generalization of Barot and Marsh's results {\cite[Lemma 2.5]{BM15}}.
\begin{lemma} \label{diagrams with frozen}
Let $Q$ be a $\Delta$ quiver and $Q'=\mu_{k}(Q)$ the mutation of $Q$ at vertex $k$. We list various types of induced subquivers in $Q$ (on the left) and corresponding cycles $C'$ in $Q'$ (on the right) arising from the mutation of $Q$ at $k$. The diagrams are drawn so that $C'$ is always a clockwise cycle in Figures \ref{Induced subquivers and the corresponding chordless cycles 1} and \ref{the corresponding chordless cycles 2}. Then every chordless cycle in $Q'$ arises in such a way.
\end{lemma}
\begin{proof}
Let $C'$ be a chordless cycle in $Q'$. We will divide $C'$ into two classes, depending on whether or not it contains the frozen vertex $\varepsilon$.

(1) If $C'$ does not contain the frozen vertex $\varepsilon$, then $C'$ is one of (a)--(g), which follows from Proposition \ref{FZ oriented cycle} and Lemma 2.5 in \cite{BM15}.
 
(2) If $C'$ contains the frozen vertex $\varepsilon$, then $C'$ must be a cycle of length 3 or 4. Otherwise, $C'$ is a cycle of length at least 5. By Proposition \ref{FZ oriented cycle}, all edges of $C'$ have trivial weights and $C'$ is mutation equivalent to the following subquiver
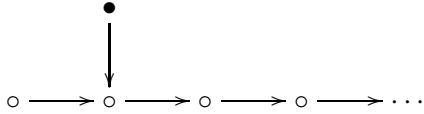
\[
\xymatrix{
& \bullet \ar[d] & & &  \\
\circ \ar[r]  & \circ \ar[r]  &\circ \ar[r]  & \circ \ar[r] & \cdots}.
\]
If $\varepsilon$ has two incident arrows, then by Proposition \ref{FZ oriented cycle}, either both of them have trival weight or one of them has weigh $2$ and the other has weight 1. This is because if $\varepsilon$ has two incident arrows with weight 2, then $C'$ is mutation equivalent to $\xymatrix{\circ \ar[r] & \circ \ar[r]^2 & \bullet}$.
\end{proof}

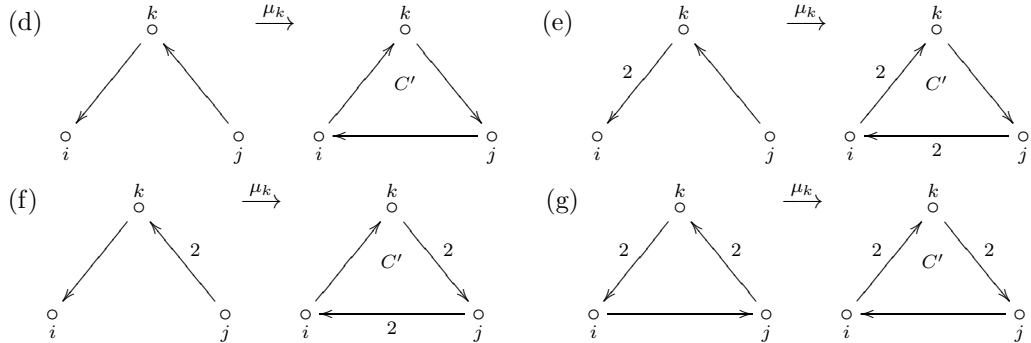
\begin{figure}[H]
\begin{itemize}
\item[(a)] The vertex $k$ does not connect to an oriented chordless cycle $C$ in $Q$. Then $C$ is the corresponding cycle in $Q'$.
 \item[(b)] The vertex $k$ connects to one vertex of an oriented chordless cycle $C$ in $Q$ (via an edge of unspecified weight). Then $C$ is the corresponding cycle in $Q'$.
\item [(c)] The vertex $k$ is one vertex of an oriented $d$-cycle without $\varepsilon$, $d\geq 4$, in $Q$. Then the local picture of $\mu_k(Q)$ becomes a $3$-cycle and a $(d-1)$-cycle, and they share a common arrow, but the orientations of two cycles are opposite, or vice versa. 
\end{itemize}
\begin{align*}
\begin{array}{lc}
\resizebox{.9\width}{.9\height}{
(d) \xymatrix{
& \overset{k}{\circ}  \ar[dl] &  \\
\underset{i}{\circ}  &  &  \underset{j}{\circ} \ar[ul]} $\overset{\mu_{k}}{\longrightarrow}$
\xymatrix{
& \overset{k}{\circ} \ar[dr]  &  \\
\underset{i}{\circ} \ar[ur] & \ar @{} [u] |{C'}  & \underset{j}{\circ} \ar[ll]}} & \resizebox{.9\width}{.9\height}{
(e) \xymatrix{
& \overset{k}{\circ} \ar@{->}[dl]_{2}&  \\
\underset{i}{\circ} &  & \underset{j}{\circ} \ar[ul] } $\overset{\mu_{k}}{\longrightarrow}$
\xymatrix{
& \overset{k}{\circ} \ar[dr] &  \\
\underset{i}{\circ} \ar@{->}[ur]^{2} &  \ar @{} [u] |{C'} & \underset{j}{\circ} \ar@{->}[ll]^{2}}} \\
\resizebox{.9\width}{.9\height}{
(f)\xymatrix{
& \overset{k}{\circ} \ar[dl] &  \\
\underset{i}{\circ} &  & \underset{j}{\circ} \ar@{->}[ul]_{2}} $\overset{\mu_{k}}{\longrightarrow}$
\xymatrix{
& \overset{k}{\circ} \ar@{->}[dr]^{2} &  \\
\underset{i}{\circ} \ar[ur] &  \ar @{} [u] |{C'} &\underset{j}{\circ}\ar@{->}[ll]^{2}}}  & \resizebox{.9\width}{.9\height}{
(g)\xymatrix{
& \overset{k}{\circ} \ar@{->}[dl]_{2} &  \\
\underset{i}{\circ}  \ar[rr] &  & \underset{j}{\circ} \ar@{->}[ul]_{2}} $\overset{\mu_{k}}{\longrightarrow}$
\xymatrix{
& \overset{k}{\circ} \ar[dr]^{2} &  \\
\underset{i}{\circ} \ar@{->}[ur]^{2} &  \ar @{} [u] |{C'} & \underset{j}{\circ} \ar[ll]}}
\end{array}
\end{align*}
\caption{Induced subquivers and the corresponding chordless cycles, part 1.}\label{Induced subquivers and the corresponding chordless cycles 1}
\end{figure}

%In order to describe the relation (R4) in Definition \ref{definition relation}, we need the following lemma:
%\begin{lemma}
%Let $Q$ be a $B^\varepsilon_k$ (respectively, $D^\varepsilon_k$) quiver. The shape of these quivers appearing in $\mathcal{U}(Q)$  are these diagrams shown in the first column of Table \ref{Boolean reflection monoids II}.
%\end{lemma}
%\begin{proof}
%The first two quivers shown in the first column of Table \ref{Boolean reflection monoids II} are mutation equivalent each other and they are in $\mathcal{U}(B^\varepsilon_n)$. From Proposition \ref{FZ oriented cycle} and Lemma \ref{diagrams with frozen}, we deduce that the quivers in $\mathcal{U}(B^\varepsilon_n)$ have only these two different types. The third quiver is in $\mathcal{U}(D^\varepsilon_n)$ by mutating vertices sequence $i_2,i_3,\ldots,i_{d-2}$. It is easy to check that the remaining quivers are in $\mathcal{U}(D^\varepsilon_n)$.  By Theorem 3.1 in \cite{V10} or using Schiffler's geometric model of cluster categories of type $D_n$ \cite{Schi08}, we deduce that the quivers in $\mathcal{U}(D^\varepsilon_n)$ have only these four different types.      
%\end{proof}

\begin{figure}
\begin{align*}
\begin{array}{lc}
\resizebox{.8\width}{.8\height}{
(a')\xymatrix{
& \overset{k}{\circ} \ar[dl] &  \\
 \underset{\varepsilon}{\bullet} &  & \underset{i}{\circ}  \ar[ul] }  $\overset{\mu_{k}}{\longrightarrow}$
\xymatrix{
& \overset{k}{\circ} \ar[dr] &  \\
\underset{\varepsilon}{\bullet} \ar[ur] &  \ar @{} [u] |{C'} & \underset{i}{\circ}  \ar[ll]}} & \resizebox{.8\width}{.8\height}{
(b')\xymatrix{
& \overset{k}{\circ} \ar[dl] &  \\
\underset{i}{\circ} &  &  \underset{\varepsilon}{\bullet} \ar[ul] }  $\overset{\mu_{k}}{\longrightarrow}$
\xymatrix{
& \overset{k}{\circ} \ar[dr] &  \\
\underset{i}{\circ} \ar[ur] &  \ar @{} [u] |{C'}  & \underset{\varepsilon}{\bullet} \ar[ll]}} \\
\resizebox{.8\width}{.8\height}{
(c')\xymatrix{
& \overset{k}{\circ} \ar[dl] &  \\
\underset{\varepsilon}{\bullet}  &  & \underset{i}{\circ}\ar@{->}[ul]_{2}}  $\overset{\mu_{k}}{\longrightarrow}$
\xymatrix{
& \overset{k}{\circ} \ar[dr]^{2} &  \\
\underset{\varepsilon}{\bullet}\ar[ur] &  \ar @{} [u] |{C'} & \underset{i}{\circ} \ar[ll]^{2}}} & \resizebox{.8\width}{.8\height}{
(d')\xymatrix{
& \overset{k}{\circ} \ar@{->}[dl]_{2} &  \\
\underset{i}{\circ}  &  &  \underset{\varepsilon}{\bullet}\ar@{->}[ul]}  $\overset{\mu_{k}}{\longrightarrow}$
\xymatrix{
& \overset{k}{\circ} \ar@{->}[dr] &  \\
\underset{i}{\circ} \ar@{->}[ur]^{2} & \ar @{} [u] |{C'} & \underset{\varepsilon}{\bullet} \ar@{->}[ll]^{2}}} \\
\resizebox{.8\width}{.8\height}{
(e')\xymatrix{
& \overset{k}{\circ} \ar@{->}[dl]_{2} &  \\
\underset{i}{\circ} \ar@{->}[rr] & & \underset{\varepsilon}{\bullet} \ar[ul]_{2}} $\overset{\mu_{k}}{\longrightarrow}$
\xymatrix{
& \overset{k}{\circ} \ar@{->}[dr]^{2} &  \\
\underset{i}{\circ} \ar@{->}[ur]^{2} &  \ar @{} [u] |{C'} & \underset{\varepsilon}{\bullet} \ar@{->}[ll]}} & \resizebox{.8\width}{.8\height}{
(f')\xymatrix{
& \overset{k}{\circ} \ar@{->}[dl]_{2} &  \\
\underset{\varepsilon}{\bullet} \ar@{->}[rr] & & \underset{i}{\circ} \ar[ul]_{2}} $\overset{\mu_{k}}{\longrightarrow}$
\xymatrix{
& \overset{k}{\circ} \ar@{->}[dr]^{2} &  \\
\underset{\varepsilon}{\bullet} \ar@{->}[ur]^{2} &  \ar @{} [u] |{C'} & \underset{i}{\circ} \ar@{->}[ll]}} \\
\resizebox{.8\width}{.8\height}{
(g')\xymatrix{
\overset{j}{\circ} \ar@{->}[drr] & & \overset{k}{\circ}  \ar@{->}[ll] \\
\underset{i}{\circ} \ar@{->}[u] & & \underset{\varepsilon}{\bullet}\ar@{->}[ll] \ar@{->}[u]}
$\overset{\mu_{k}}{\longrightarrow}$
\xymatrix{
\overset{j}{\circ}  \ar@{->}[rr] & & \overset{k}{\circ} \ar@{->}[d]  \\
\underset{i}{\circ} \ar@{->}[u] &  \ar @{} [u] |{C'} & \underset{\varepsilon}{\bullet}\ar@{->}[ll] }} & \resizebox{.8\width}{.8\height}{
(h')\xymatrix{
\overset{j}{\circ}  \ar@{->}[rr] & & \overset{k}{\circ} \ar@{->}[d]  \\
\underset{i}{\circ} \ar@{->}[u] &  & \underset{\varepsilon}{\bullet}\ar@{->}[ll] }
$\overset{\mu_{k}}{\longrightarrow}$
\xymatrix{
\overset{j}{\circ} \ar@{->}[drr] & & \overset{k}{\circ}  \ar@{->}[ll] \\
\ar @{} [ur] |{C'} \underset{i}{\circ} \ar@{->}[u] &  & \underset{\varepsilon}{\bullet}\ar@{->}[ll] \ar@{->}[u]}}\\ 
\resizebox{.8\width}{.8\height}{
(i') \xymatrix{
\overset{j}{\circ} \ar@{->}[rr] & & \overset{i}{\circ} \ar@{->}[d] \\
\underset{k}{\circ} \ar@{<-}[u] \ar@{->}[rr]  & & \underset{\varepsilon}{\bullet}\ar@{->}[ull]}\ $\overset{\mu_{k}}{\longrightarrow}$ \
\xymatrix{
\overset{j}{\circ}\ar@{->}[rr] & & \overset{i}{\circ} \ar@{->}[d]  \\
\underset{k}{\circ} \ar@{->}[u] &  \ar @{} [u] |{C'} &  \underset{\varepsilon}{\bullet}\ar@{->}[ll]}} & \resizebox{.8\width}{.8\height}{
(j') \xymatrix{
\overset{j}{\circ}\ar@{->}[rr] & & \overset{i}{\circ} \ar@{->}[d]  \\
\underset{k}{\circ} \ar@{->}[u] & &  \underset{\varepsilon}{\bullet}\ar@{->}[ll]} \ $\overset{\mu_{k}}{\longrightarrow}$ \
\xymatrix{
\overset{j}{\circ} \ar@{->}[rr] & & \overset{i}{\circ} \ar@{->}[d] \\
\underset{k}{\circ} \ar@{<-}[u] \ar@{->}[rr]  & &  \ar @{} [ul] |{C'} \underset{\varepsilon}{\bullet}\ar@{->}[ull]}} \\
\resizebox{.8\width}{.8\height}{
(k')\xymatrix{
\overset{k}{\circ} \ar@{->}[d] & & \overset{i}{\circ}  \ar@{->}[ll] \ar@{->}[d] \\
\underset{j}{\circ} \ar@{->}[urr] & & \underset{\varepsilon}{\bullet}\ar@{->}[ll]} $\overset{\mu_{k}}{\longrightarrow}$
\xymatrix{
\overset{k}{\circ}\ar@{->}[rr] & & \overset{i}{\circ} \ar@{->}[d]  \\
\underset{j}{\circ} \ar@{->}[u] &   \ar @{} [u] |{C'} &  \underset{\varepsilon}{\bullet}\ar@{->}[ll]}}  &  \resizebox{.8\width}{.8\height}{
(l') \xymatrix{
\overset{k}{\circ}\ar@{->}[rr] & & \overset{i}{\circ} \ar@{->}[d]  \\
\underset{j}{\circ} \ar@{->}[u] &  &  \underset{\varepsilon}{\bullet}\ar@{->}[ll]}$\overset{\mu_{k}}{\longrightarrow}$ \xymatrix{
\overset{k}{\circ} \ar@{->}[d] & & \overset{i}{\circ}  \ar@{->}[ll] \ar@{->}[d] \\
\underset{j}{\circ} \ar@{->}[urr] & & \ar @{} [ul] |{C'} \underset{\varepsilon}{\bullet}\ar@{->}[ll]}}\\
\end{array}
\end{align*}
\caption{Induced subquivers and the corresponding chordless cycles, part 2.}\label{the corresponding chordless cycles 2}
\end{figure}
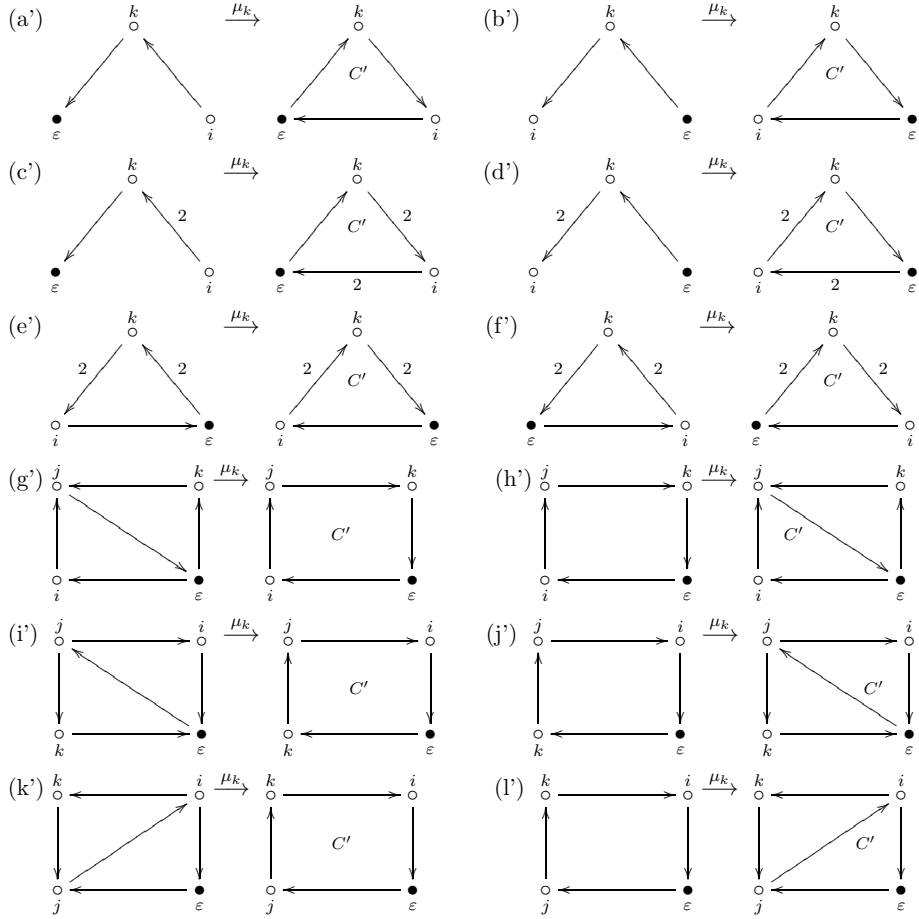

\section{Cycle relations and path relations}\label{cycle relations and path relations}

In this section, we find an efficient subset of the relations to define Boolean reflection monoids, which generalizes Barot and Marsh's results, Lemmas 4.1, 4.2, 4.4 and Proposition 4.6 in \cite{BM15}.

Let $\Delta$ be one of $A^\varepsilon_{n-1}$, $B^\varepsilon_{n}$, and $D^\varepsilon_{n}$ in Table \ref{initial ABD}. Suppose that $Q$ is any quiver that is mutation equivalent to a $\Delta$ quiver. The following lemma gives an efficient subset of relations (R3) and (R4) in Definition \ref{definition relation}.

\begin{lemma}\label{a chordless cycle lemma in M}
Let $M(Q)$ be an inverse monoid with generators subject to relations (R1) and (R2) in Definition \ref{definition relation}.

\begin{itemize}
\item[(1)] If $Q$ contains a chordless cycle $C'_3$, see Figure \ref{lemma 5.2}, then the following statements are equivalent:
\begin{itemize}
  \item[(a)] $s_\varepsilon s_1 s_2 s_1 = s_1 s_2 s_1 s_\varepsilon$;
  \item[(b)] $s_1s_2s_\varepsilon s_2=s_2 s_\varepsilon s_2 s_1$.
\end{itemize}
Furthermore, if one of the above holds, then the following statements are equivalent:
\begin{itemize}
  \item[(c)] $s_\varepsilon s_1 s_\varepsilon = s_\varepsilon s_1 s_\varepsilon s_1 = s_1 s_\varepsilon s_1 s_\varepsilon$;
  \item[(d)] $s_\varepsilon s_2 s_\varepsilon = s_\varepsilon s_2 s_\varepsilon s_2=s_2 s_\varepsilon s_2 s_\varepsilon$.
\end{itemize}
\item[(2)] If $Q$ contains a chordless cycle $C''_3$, see Figure \ref{lemma 5.2}, then $s_1s_2s_\varepsilon s_2=s_2 s_\varepsilon s_2 s_1$.
\item[(3)] If $Q$ contains a chordless cycle $C'_4$, see Figure \ref{lemma 5.2}, then the following statement holds:
\begin{itemize}
  \item[(a)] $s_\varepsilon s_1 s_2 s_3 s_2 s_1 = s_1 s_2 s_3 s_2 s_1 s_\varepsilon$;
  \item[(b)] $s_1 s_2 s_3 s_\varepsilon s_3 s_2 = s_2 s_3 s_\varepsilon s_3 s_2 s_1$.
\end{itemize}
Furthermore, if one of the above holds, then the following statements are equivalent:
\begin{itemize}
  \item[(c)] $s_\varepsilon s_1 s_\varepsilon = s_\varepsilon s_1 s_\varepsilon s_1 = s_1 s_\varepsilon s_1 s_\varepsilon$;
  \item[(d)] $s_\varepsilon s_3 s_\varepsilon = s_\varepsilon s_3 s_\varepsilon s_3=s_3 s_\varepsilon s_3 s_\varepsilon$.
\end{itemize}
\item[(4)] If $Q$ contains a subquiver $C'_d$, see Figure \ref{Type IV}, then the following statements are equivalent:
\begin{itemize}\label{conditions 12}
  \item[(a)] $s_{i_1} P(s_{c'}, s_\varepsilon)P(s_\varepsilon,s_{c'})s_{i_1}= s_{i_2}\cdots s_{i_d} P(s_{c'}, s_\varepsilon)P(s_\varepsilon,s_{c'})s_{i_d}\cdots s_{i_2}$;
	\item[(b)] $s_{i_{a-1}}\cdots s_{i_1} P(s_{c'}, s_\varepsilon) P(s_\varepsilon,s_{c'}) s_{i_1} \cdots  s_{i_{a-1}}=s_{i_a} \cdots s_{i_d} P(s_{c'}, s_\varepsilon) P(s_\varepsilon,s_{c'}) s_{i_d} \cdots$ $s_{i_a}$ for any $a =3,\cdots,d$. 
\end{itemize}
\end{itemize}
\begin{figure}[H]
\begin{align*}
\xymatrix{
& \overset{2}{\circ} \ar@{->}[dr] & \\
\underset{1}{\circ} \ar@{->}[ur] & \ar @{} [u] |{C'_3} & \underset{\varepsilon}{\bullet} \ar@{->}[ll]}  \qquad    
\xymatrix{
& \overset{2}{\circ} \ar@{->}[dr] & \\
\underset{1}{\circ} \ar@{->}[ur]^{2} & \ar @{} [u] |{C^{''}_3} & \underset{\varepsilon}{\bullet} \ar@{->}[ll]^{2}} \qquad  
\xymatrix{
\overset{2}{\circ} \ar@{->}[rr] & &  \overset{3}{\circ} \ar@{->}[d] \\
\underset{1}{\circ} \ar@{->}[u] & \ar @{} [u] |{C'_4} & \underset{\varepsilon}{\bullet} \ar@{->}[ll]} 
\end{align*}
\caption{A chordless 3-cycle $C'_3$, a chordless 3-cycle $C''_3$, and a chordless 4-cycle $C'_4$. (see Lemma \ref{a chordless cycle lemma in M})}\label{lemma 5.2}
\end{figure}
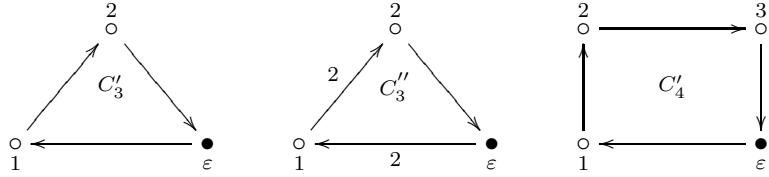
\end{lemma}

\begin{proof}
For (1), the equivalence of (a) and (b) follows from:
\begin{align*}
s_1s_2s_\varepsilon s_2=s_2(s_1 s_2s_1 s_\varepsilon) s_2, \ s_2 s_\varepsilon s_2 s_1=s_2(s_\varepsilon s_1 s_2s_1) s_2,
\end{align*}
using (R2). Suppose that (a) and (b) hold. Then by (R1), (R2), (a), and (b), the equivalence of $(c)$ and $(d)$ follows from:
\begin{align*}
s_1 s_2 s_1 (s_\varepsilon s_1 s_\varepsilon) s_1 s_2 s_1 & = s_\varepsilon s_1 s_2 s_1 s_1 s_1 s_2 s_1 s_\varepsilon = s_\varepsilon s_2 s_\varepsilon, \\
s_1 s_2 s_1 (s_\varepsilon s_1 s_\varepsilon s_1) s_1 s_2 s_1 & = s_\varepsilon (s_1 s_2 s_\varepsilon s_2) s_1 = s_\varepsilon s_2 s_\varepsilon s_2, \\
s_1 s_2 s_1 (s_1 s_\varepsilon s_1 s_\varepsilon) s_1 s_2 s_1 & = s_1 (s_2 s_\varepsilon s_2 s_1) s_\varepsilon = s_2 s_\varepsilon s_2 s_\varepsilon.
\end{align*}

For (3), the equivalence of (a) and (b) follows from:
\begin{align*}
s_1 s_2 s_3 s_\varepsilon s_3 s_2 & = s_2 s_3 (s_1 s_2 s_3 s_2 s_1 s_\varepsilon) s_3 s_2, \ s_2 s_3 s_\varepsilon s_3 s_2 s_1=s_2 s_3 (s_\varepsilon s_1 s_2 s_3 s_2 s_1) s_3 s_2,
\end{align*}
using first $s_1s_2s_3s_2s_1=s_3s_2s_1s_2s_3$ and then (R1). Suppose that (a) and (b) hold. Using first $s_2s_\varepsilon=s_\varepsilon s_2$ and then (a), we have:
\begin{align*}
s_1 s_2 s_3 s_2 s_1 s_2 (s_\varepsilon s_1 s_\varepsilon) s_2 s_1 s_2 s_3 s_2 s_1 & = (s_1 s_2 s_3 s_2 s_1 s_\varepsilon) s_2 s_1 s_2 (s_\varepsilon s_1 s_2 s_3 s_2 s_1) \\
& = (s_\varepsilon s_1 s_2 s_3 s_2 s_1) s_2 s_1 s_2 (s_1 s_2 s_3 s_2 s_1 s_\varepsilon) \\
& = s_\varepsilon s_1 s_2 s_3 s_2 s_3 s_2 s_1 s_\varepsilon \quad \text{(by (R2))} \\
& = s_\varepsilon s_3 s_\varepsilon,
\end{align*}
where in the last equation we used that $s_1$ and $s_3$ commute. Using (R1), (R2), (a), and (b), by a similar argument, we have
\begin{align*}
s_1 s_2 s_3 s_2 s_1 s_2 (s_1 s_\varepsilon s_1 s_\varepsilon) s_2 s_1 s_2 s_3 s_2 s_1 & =  s_3 s_\varepsilon s_3 s_\varepsilon. \\
s_1 s_2 s_3 s_2 s_1 s_2 (s_\varepsilon s_1 s_\varepsilon s_1) s_2 s_1 s_2 s_3 s_2 s_1 & = s_\varepsilon s_3 s_\varepsilon s_3.
\end{align*}
Therefore $(c)$ and $(d)$ are equivalent.

For (4), using $(R1)$, it is obvious.
\end{proof}

At an end of this section, we show that $M(Q)$ could be defined using only the underlying unoriented weighted diagram of $Q$, by taking relations (R1)--(R4) corresponding to both $Q$ and $Q^{op}$ as the defining relations. Our result can be viewed as a generalization of Proposition 4.6 of \cite{BM15}.

\begin{proposition}\label{without orient}
Let $M(\Phi,\mathcal{B})$ be a Boolean reflection monoid with generators $s_i$, $i\in I \cup \{ \varepsilon \}$. Then the generators satisfy (R1)--(R4) with respect to $Q$ if and only if they satisfy (R1)--(R4) with respect to ${Q}^{op}$.
\end{proposition}
\begin{proof}
We assume that generators $s_i$, $i\in I \cup \{ \varepsilon \}$ satisfy relations (R1)--(R4) with respect to $Q$, and show that these generators satisfy relations (R1)--(R4) with respect to $Q^{op}$. The converse follows by replacing $Q$ with $Q^{op}$. Since (R1) and (R2) do not depend on the orientation of $Q$, generators $s_i$, $i\in I \cup \{ \varepsilon \}$ satisfy relation (R1) and (R2) with respect to ${Q}^{op}$. The cases of chordless cycles appearing in quivers of finite type have been proved in Proposition 4.6 of \cite{BM15}. The remaining needed to check the cases are $C'_3$, $C''_3$, $C'_4$ shown in Figure \ref{lemma 5.2} and $C'_d$ shown in Figure \ref{Type IV}.

{\bf Case 1.} In $C'_3$, we have
\begin{align*}
s_\varepsilon s_2 s_1 s_2 & = s_\varepsilon s_1 s_2 s_1 = s_1 s_2 s_1 s_\varepsilon = s_2 s_1 s_2 s_\varepsilon, \\
s_2 s_1 s_\varepsilon s_1 & = s_1 s_2 (s_1 s_2 s_\varepsilon s_2) s_2 s_1  = s_1 s_2 (s_2 s_\varepsilon s_2 s_1) s_2 s_1 = s_1 s_\varepsilon s_1 s_2.
\end{align*}

{\bf Case 2.} In $C''_3$, we have
\begin{align*}
s_\varepsilon s_2 s_1 s_2 = s_2 s_1 (s_1 s_2 s_\varepsilon s_2) s_1 s_2 = s_2 s_1 (s_2 s_\varepsilon s_2 s_1)s_1 s_2 = s_2 s_1 s_2 s_\varepsilon.
\end{align*}

{\bf Case 3.} In $C'_4$, note that $s_1s_2s_3s_2s_1=s_3s_2s_1s_2s_3$. We have
\begin{align*}
s_\varepsilon s_3 s_2 s_1 s_2 s_3 & = s_\varepsilon s_1 s_2 s_3 s_2 s_1 = s_1 s_2 s_3 s_2 s_1 s_\varepsilon = s_3 s_2 s_1 s_2 s_3 s_\varepsilon,\\
s_3 s_2 s_1 s_\varepsilon s_1 s_2 & = s_2 s_1 s_3 s_2 (s_1 s_2 s_3 s_\varepsilon s_3 s_2) s_2 s_3 s_1 s_2 = s_2 s_1 s_3 s_2 (s_2 s_3 s_\varepsilon s_3 s_2 s_1) s_2 s_3 s_1 s_2 \\
&=s_2 s_1 s_\varepsilon s_1 s_2 s_3.
\end{align*}

{\bf Case 4.} In $C'_d$, it follows from Lemma \ref{a chordless cycle lemma in M} (4) that (R4) does not depend on the orientation of chordless cycles in $C'_d$.

Since every chordless cylce in $Q^{op}$ corresponds to a chordless cycle in $Q$, the result holds.
\end{proof}

\section{The proof of Theorem \ref{mutation-invariance}}\label{the proof of main theorem}
In this section, we give the proof of Theorem \ref{mutation-invariance}.

Let $\Delta$ be one of $A^\varepsilon_{n-1}$, $B^\varepsilon_{n}$, and $D^\varepsilon_{n}$ in Table \ref{initial ABD}. We fix a $\Delta$ quiver $Q$. Let $Q'=\mu_{k}(Q)$ be the mutation of $Q$ at vertex $k$, where $k \in I$. Throughout the section, we use $s_i$ and $r_i$ for $i\in I \cup\{\varepsilon\}$ to denote generators of $M(Q)$ and $M(Q')$ respectively. Similar to \cite{BM15}, we define elements $t_i$, $i \in I$, and $t_{\varepsilon}$ in $M(Q)$ as follows:
\begin{align}\label{define variables}
\begin{split}
&t_i=
\begin{cases}
s_{k}s_{i}s_{k} & \text{if there is an arrow $i \to k$ in $Q$ (possibly weighted)}, \\
s_i & \text{otherwise},
\end{cases}\\
&t_{\varepsilon}=
\begin{cases}
s_{k}s_{\varepsilon} s_{k} & \text{if there is an arrow $\varepsilon \to k$ in $Q$ (possibly weighted)}, \\
s_{\varepsilon} & \text{otherwise}.
\end{cases}
\end{split}
\end{align}
Then $t^{2}_i=e$ for $i\in I$ and $t_{\varepsilon}^2=t_{\varepsilon}$. In order to prove Theorem \ref{mutation-invariance}, we need the following proposition, which we will prove it in Section \ref{proof of 7.1}.

%\begin{align}\label{idempotent property}
%\begin{split}
%&t^{2}_i=
%\begin{cases}
%(s_{k}s_{i}s_{k})(s_{k}s_{i}s_{k})=e & \text{if there is an arrow $i \to k$ in $Q$ (possibly weighted)}, \\
%s^{2}_i = e & \text{otherwise},
%\end{cases}\\
%&t_{\varepsilon}^2=\begin{cases}
%(s_{k}s_{\varepsilon} s_{k})(s_{k}s_{\varepsilon}s_{k})= t_{\varepsilon} & \text{if there is an arrow $\varepsilon \to k$ in $Q$ (possibly weighted)}, \\
%s_{\varepsilon}^2 = t_{\varepsilon} & \text{otherwise}.
%\end{cases}
%\end{split}
%\end{align}

\begin{proposition}\label{inverse monoids homomorphism}
For each $i\in I\cup \{\varepsilon\}$, the map
\begin{align*}
\Phi:M(Q') & \longrightarrow M(Q) \\
r_i & \longmapsto t_i,
\end{align*}
is an inverse monoid homomorphism.
\end{proposition}

\subsection{Proof of Theorem \ref{mutation-invariance}}
For each vertex $i\in I \cup \{\varepsilon\}$ of $Q$ define the elements $t'_i$ in $M(Q')$ as follows:
\begin{align*}
\begin{split}
& t'_i =
\begin{cases}
r_{k}r_{i}r_{k} & \text{if there is an arrow $k \to i$ in $Q'$ (possibly weighted)}, \\
r_i & \text{otherwise},
\end{cases}\\
& t'_\varepsilon =
\begin{cases}
r_{k}r_{\varepsilon} r_{k} & \text{if there is an arrow $k \to \varepsilon$ in $Q'$ (possibly weighted)}, \\
r_{\varepsilon} & \text{otherwise}.
\end{cases}
\end{split}
\end{align*}
We claim that these elements $t'_i$, for each vertex $i\in I \cup \{\varepsilon\}$, satisfy the relations $(R1)$--$(R4)$ defining $M(Q)$. This follows from Proposition \ref{inverse monoids homomorphism} by interchanging $Q$ and $Q'$ and using the fact that the definition of $M(Q)$ is unchanged under reversing the orientation of all the arrows in $Q$ (see Proposition \ref{without orient}). Therefore there is an inverse monoid homomorphism $\Theta:M(Q)\to M(Q')$ such that $\Theta(s_i)=t'_i$ for each $i$.

If there is no arrow $i \to k$ in $Q$, then there is also no arrow $k \to i$ in $Q'$ and consequently $\Theta\circ\Phi(r_i)=\Theta(s_i)=r_i$. If there is an arrow $i \to k$ in $Q$, then there is an arrow $k \to i$ in $Q'$ and therefore $\Theta\circ\Phi(r_i)=\Theta(s_ks_is_k)=\Theta(s_k)\Theta(s_i)\Theta(s_k)=r_k(r_{k}r_{i}r_{k})r_k=r_i$.
So $\Theta\circ\Phi=id_{M(Q')}$, and, similarly, $\Phi\circ\Theta=id_{M(Q)}$, and hence $\Theta$ and $\Phi$ are isomorphisms.

\subsection{The proof of Proposition \ref{inverse monoids homomorphism}}\label{proof of 7.1}
We will prove Proposition \ref{inverse monoids homomorphism} by showing that the elements $t_i, i \in I\cup  \{\varepsilon \}$ satisfy the (R1)--(R4) relations in $M(Q')$. We denote by $m_{ij}'$ the value of $m_{ij}$ for $Q'$. (R1) is obvious. In the sequel, the proof that the elements $t_i, i \in I\cup  \{\varepsilon \}$, satisfy (R2) in $M(Q')$ follows from Lemma \ref{proof of (R2)} and the rest of proof is completed case by case.

\begin{lemma}\label{proof of (R2)}
The elements $t_i$, for $i$ a vertex of $Q$, satisfy the following relations.
\begin{enumerate}
\item If $i=k$ or $j=k$ and $i, j \neq \varepsilon$, then $(t_it_j)^{m_{ij}'}=e$.
\item If at most one of $i,j$ is connected to $k$ in $Q$ and $i, j \neq \varepsilon$, then $(t_it_j)^{m_{ij}'}=e$.
\item Let $i$ be in $I$. Then
\begin{align*}
\begin{cases}
t_i t_\varepsilon = t_\varepsilon t_i & \text{if $\varepsilon, i$ are not connected in $Q'$,}\\
t_\varepsilon t_i t_\varepsilon = t_\varepsilon t_i t_\varepsilon t_i = t_i t_\varepsilon t_i t_\varepsilon & \text{if $\varepsilon, i$ are connected by an edge with weight $1$ in $Q'$,} \\
t_\varepsilon = t_i t_\varepsilon= t_\varepsilon t_i & \text{if $\varepsilon, i$ are connected by an edge with weight $2$ in $Q'$.} \\
\end{cases}
\end{align*}
\end{enumerate}
\end{lemma}

\begin{proof}
In Lemma 5.1 of \cite{BM15}, Barot and Marsh proved the parts (1) and (2). We only need to prove the part (3).

Suppose without loss of generality that $i =k$. The only nontrivial case is when there is an arrow $\varepsilon\to k=i$ with a weight $q$ in $Q$. If $q=1$, then 
\begin{gather}
\begin{align*}
t_\varepsilon t_k t_\varepsilon t_k = (s_k s_\varepsilon s_k) s_k (s_k s_\varepsilon s_k) s_k =\begin{cases}
s_k s_\varepsilon s_k s_\varepsilon = s_\varepsilon s_k s_\varepsilon s_k = s_k (s_k s_\varepsilon s_k) s_k (s_k s_\varepsilon s_k)
= t_k t_\varepsilon t_k t_\varepsilon, \\ 
s_\varepsilon s_k s_\varepsilon = s_k s_\varepsilon s_k s_\varepsilon s_k = (s_k s_\varepsilon s_k) s_k (s_k s_\varepsilon s_k)=t_\varepsilon t_k t_\varepsilon.
\end{cases}
%t_\varepsilon t_k t_\varepsilon & = (s_k s_\varepsilon s_k) s_k (s_k s_\varepsilon s_k)
%=  s_k s_\varepsilon s_k s_\varepsilon s_k = s_\varepsilon s_k s_\varepsilon = s_k s_\varepsilon s_k s_\varepsilon = t_\varepsilon t_k t_\varepsilon t_k.
\end{align*}
\end{gather}
If $q=2$, note that $s_\varepsilon s_k=s_k s_\varepsilon=s_\varepsilon$, then
\[t_k t_\varepsilon= s_k (s_k s_\varepsilon s_k) = s_\varepsilon s_k = s_k s_\varepsilon = s_\varepsilon = t_\varepsilon t_k = t_\varepsilon.\]

Suppose that $i \neq k$. We divide this proof into three cases.

{\bf Case 1.} There are no arrows from $i,\varepsilon$ to $k$, then $t_i =s_i$, $t_\varepsilon=s_\varepsilon$ hold $(3)$.

{\bf Case 2.} There are arrows from one of $i, \varepsilon$ to $k$ and there are no arrows from the other of $i, \varepsilon$ to $k$ in $Q$, then we assume that there are arrows from $\varepsilon$ to $k$ and there are no arrows from $i$ to $k$ in $Q$. If $\varepsilon, i$ are not connected in $Q$, then $t_i t_\varepsilon = s_{i} (s_{k} s_{\varepsilon} s_{k})= (s_{k} s_{\varepsilon} s_{k})s_{i}= t_\varepsilon t_i$.
If $\varepsilon, i$ are connected by an edge with weight $1$ in $Q$, then
\begin{gather}
\begin{align*}
t_\varepsilon t_i t_\varepsilon = (s_{k} s_{\varepsilon} s_{k}) s_{i} (s_{k} s_{\varepsilon} s_{k}) = s_{k} s_{\varepsilon} s_{i} s_{\varepsilon} s_{k} & = 
\begin{cases} 
s_{k} s_{\varepsilon} s_{i} s_{\varepsilon} s_{k}s_{i}
= (s_{k} s_{\varepsilon} s_{k}) s_{i} (s_{k} s_{\varepsilon} s_{k})s_{i} =t_\varepsilon t_i t_\varepsilon t_i, \\
s_{i} s_{k} s_{\varepsilon} s_{i} s_{\varepsilon} s_{k} = s_{i} (s_{k} s_{\varepsilon} s_{k}) s_{i}(s_{k} s_{\varepsilon} s_{k})
= t_i t_\varepsilon t_i t_\varepsilon.
\end{cases}
\end{align*}
\end{gather}
That $\varepsilon, i$ are connected by an edge with weight $2$ and there are no arrows from $i$ to $k$ in $Q$ is impossible.

%, because of the fact that there is only 3-chordless cycle in the mutation class of $B^{\varepsilon}_n$ quivers and Corollary \ref{location of frozen vertex}.

%If $\varepsilon, i$ are connected by an edge with weight $2$ in $Q$, then
%\begin{align*}
%\hspace{1.5cm} t_\varepsilon = s_{k} s_{\varepsilon} s_{k} & = s_{k} s_i s_{\varepsilon} s_{k} = s_i s_{k} s_{\varepsilon} s_{k}
%= t_i t_\varepsilon \\
%& = s_{k} s_{\varepsilon} s_i s_{k} = s_{k} s_{\varepsilon} s_{k} s_i = t_\varepsilon t_i.
%\end{align*}

{\bf Case 3.} There are arrows from $i, \varepsilon$ to $k$. The possibilities for the subquivers induced by $i$, $\varepsilon$, and $k$ are enumerated in (a')--(g') of Figure \ref{three vertices subdiagram mutation 1}. We show that $t_i$ and $t_\varepsilon$ satisfy (3) by checking each case. Within each case, subcase (i) is when the subquiver of $Q$ is the diagram on the left, and subcase (ii) is when the subquiver of $Q$ is the diagram on the right.

$(a')$(i) We have $t_it_\varepsilon = (s_k s_i s_k)(s_k s_\varepsilon s_k)=s_k s_i s_\varepsilon s_k=s_k s_\varepsilon s_i s_k = (s_k s_\varepsilon s_k)(s_k s_i s_k)=t_\varepsilon t_i$.

$(a')$(ii) We have $t_it_\varepsilon =s_i s_\varepsilon=s_\varepsilon  s_i=t_\varepsilon t_i$.

$(b')$ (i) We have 
\begin{align*}
t_\varepsilon t_{i} t_\varepsilon & = s_\varepsilon (s_{k} s_i s_k) s_\varepsilon=s_\varepsilon s_{i} s_k s_i s_\varepsilon=s_{i} (s_\varepsilon s_k s_\varepsilon) s_i, \\
& = \begin{cases}
s_i s_\varepsilon s_k s_\varepsilon s_k s_i= s_\varepsilon s_{i} s_k s_i s_\varepsilon s_{i} s_k s_i = s_\varepsilon (s_k s_i s_k) s_\varepsilon (s_k s_i s_k) = t_\varepsilon t_{i} t_\varepsilon t_{i}, \\
s_i s_k s_\varepsilon s_k s_\varepsilon s_i = s_i s_k s_i s_\varepsilon  s_i s_k s_i s_\varepsilon= (s_{k} s_i s_k) s_\varepsilon (s_{k} s_i s_k) s_\varepsilon =t_{i} t_\varepsilon t_{i} t_\varepsilon.
\end{cases}
\end{align*}

$(b')$ (ii) We have $t_i t_\varepsilon=s_i (s_k s_\varepsilon s_k)=s_k(s_k s_i s_k s_\varepsilon)s_k =s_k(s_\varepsilon s_k s_i s_k)s_k=(s_k s_\varepsilon s_k)s_i=t_\varepsilon t_i$.

$(c')$ (i) We have
\begin{align*}
t_\varepsilon t_i t_\varepsilon & = (s_k s_\varepsilon s_k) s_i (s_k s_\varepsilon s_k)=s_k s_\varepsilon s_i s_k s_i s_\varepsilon s_k=s_k s_i s_\varepsilon s_k s_\varepsilon s_i s_k, \\
& = \begin{cases}
s_k s_i s_\varepsilon s_k s_\varepsilon s_k s_i s_k = s_k s_\varepsilon s_i s_k s_i s_\varepsilon s_k s_i = (s_k s_\varepsilon s_k) s_i (s_k s_\varepsilon s_k) s_i = t_\varepsilon t_{i} t_\varepsilon t_{i}, \\
s_k s_i  s_k s_\varepsilon s_k s_\varepsilon s_i s_k =s_i s_k s_\varepsilon s_i s_k s_i s_\varepsilon s_k = s_i (s_k s_\varepsilon s_k) s_i (s_k s_\varepsilon s_k) = t_{i} t_\varepsilon t_{i} t_\varepsilon.
\end{cases}
\end{align*}

$(c')$ (ii) We have $t_i t_\varepsilon=(s_k s_i s_k)s_\varepsilon=s_i s_k s_i s_\varepsilon =s_\varepsilon s_i s_k s_i=s_\varepsilon s_k s_i s_k=t_\varepsilon t_i$.

$(d')$ (i) We have 
$$t_i t_\varepsilon=(s_k s_i s_k)(s_k s_\varepsilon s_k)=s_k s_i s_\varepsilon s_k=s_k s_\varepsilon s_i s_k=(s_k s_\varepsilon s_k)(s_k s_i s_k)=t_\varepsilon t_i.$$

$(d')$(ii) We have $t_i t_\varepsilon=s_i s_\varepsilon=s_\varepsilon s_i=t_\varepsilon t_i$.

$(e')$(i) Note that $s_i s_k s_\varepsilon=s_k s_\varepsilon$ and $s_\varepsilon s_k s_i=s_\varepsilon s_k$. We have
\begin{align*}
t_i t_\varepsilon&=(s_k s_i s_k) s_\varepsilon =s_k (s_k s_\varepsilon) = s_\varepsilon = t_\varepsilon, \\
t_\varepsilon t_i&=s_\varepsilon (s_k s_i s_k) = (s_\varepsilon s_k) s_k=s_\varepsilon = t_\varepsilon.
\end{align*}

$(e')$(ii) We have 
$$t_it_\varepsilon=s_i (s_k s_\varepsilon s_k)=s_i s_k (s_\varepsilon s_k s_i s_k)s_k s_i=s_i s_k (s_k s_i s_ks_\varepsilon)s_k s_i=s_ks_\varepsilon s_k s_i=t_\varepsilon t_i.$$

$(f')$(i) Note that $s_i s_k s_\varepsilon = s_k s_\varepsilon$ and $s_\varepsilon s_k s_i=s_\varepsilon s_k$. We have
\begin{align*}
t_i t_\varepsilon = s_i(s_k s_\varepsilon s_k)=s_k s_\varepsilon s_k= t_\varepsilon, \
t_\varepsilon t_i = (s_k s_\varepsilon s_k) s_i =s_k s_\varepsilon s_k=t_\varepsilon .
\end{align*}

$(f')$(ii) We have $t_it_\varepsilon=(s_k s_i s_k) s_\varepsilon= s_k (s_i s_k s_\varepsilon s_k) s_k =s_k (s_k s_\varepsilon s_k s_i) s_k = s_\varepsilon s_k s_is_k=t_\varepsilon t_i$.

$(g')$(i) Note that $s_k s_\varepsilon = s_\varepsilon s_k=s_\varepsilon$. We have
\begin{align*}
t_\varepsilon t_i t_\varepsilon & = (s_k s_\varepsilon s_k) s_i (s_k s_\varepsilon s_k) = s_\varepsilon s_i s_\varepsilon
= \begin{cases}
s_i s_\varepsilon s_i s_\varepsilon = t_i t_\varepsilon t_i t_\varepsilon,\\
s_\varepsilon s_i s_\varepsilon s_i= t_\varepsilon t_i t_\varepsilon t_i.
\end{cases}
\end{align*}

$(g')$(ii) Note that $s_k s_\varepsilon = s_\varepsilon s_k =s_\varepsilon$. We have
\begin{align*}
t_\varepsilon t_i t_\varepsilon & = s_\varepsilon (s_k s_i s_k) s_\varepsilon=
\begin{cases}
s_\varepsilon s_i  s_\varepsilon s_k = s_\varepsilon s_i  s_\varepsilon s_i s_k = s_\varepsilon (s_k s_i s_k) s_\varepsilon (s_k s_i s_k) = t_\varepsilon t_i t_\varepsilon t_i, \\
s_k s_\varepsilon s_i s_\varepsilon = s_k s_i s_\varepsilon s_i s_\varepsilon= (s_k s_i s_k) s_\varepsilon (s_k s_i s_k) s_\varepsilon = t_i t_\varepsilon t_i t_\varepsilon.
\end{cases}
\end{align*}
\end{proof}

The possibilities for chordless cycles in $\mathcal{U}(\Delta)$ are enumerated in Lemma \ref{diagrams with frozen}. For (R3), $(a)$ is trivial and $(b)$ follows from the commutative property of $t_k$ and $t_i$, where $i$ is not incident to $k$. Barot and Marsh proved in \cite{BM15} that (R3)~(i) holds for (a)--(g). It is enough to show that (R3)~(ii) and (R3)~(iii) hold by checking $(a')$--$(l')$. In each case, we need to check that the corresponding cycle relations hold. In the sequel, we frequently use (R1) and (R2) without comment.

$(a')$ We have $t_\varepsilon t_k t_i t_k =  s_\varepsilon s_k (s_k s_i s_k) s_k = s_\varepsilon s_i = s_i s_\varepsilon = s_k (s_k s_i s_k) s_k s_\varepsilon  = t_k t_i t_k t_\varepsilon.$

$(b')$ We have $t_\varepsilon t_i t_k t_i = (s_k s_\varepsilon s_k) s_i s_k s_i = s_k s_\varepsilon s_i s_k = s_k s_i s_\varepsilon s_k=s_i s_k s_i (s_k s_\varepsilon s_k)=t_i t_k t_i t_\varepsilon$.

$(c')$ We have $t_\varepsilon t_k t_i t_k =  s_\varepsilon s_k (s_k s_i s_k) s_k = s_\varepsilon s_i = s_i s_\varepsilon = s_k (s_k s_i s_k) s_k s_\varepsilon  = t_k t_i t_k t_\varepsilon.$

$(d')$ We have $t_i t_k t_\varepsilon t_k = s_i s_k (s_k s_\varepsilon s_k) s_k  = s_i s_\varepsilon = s_\varepsilon s_i = s_k (s_k s_\varepsilon s_k) s_k s_i = t_k t_\varepsilon t_k t_i.$

$(e')$ Note that $s_k s_\varepsilon= s_\varepsilon s_k = s_\varepsilon$ and $s_k s_i s_\varepsilon s_i=s_i s_\varepsilon s_is_k$. We have
\begin{align*}
t_\varepsilon t_i t_k t_i & = (s_k s_\varepsilon s_k) s_i s_k s_i = s_\varepsilon s_i s_k s_i= s_i s_k s_i s_\varepsilon = s_i s_k s_i (s_k s_\varepsilon s_k) = t_i t_k t_i t_\varepsilon.
\end{align*}

$(f')$ Note that $s_k s_\varepsilon= s_\varepsilon s_k = s_\varepsilon$ and $s_\varepsilon s_i s_k s_i=s_i s_k s_i s_\varepsilon$. We have
\begin{align*}
t_k t_i t_\varepsilon t_i & = s_k (s_k s_i s_k) s_\varepsilon (s_k s_i s_k) = s_i s_\varepsilon s_i s_k = s_k s_i s_\varepsilon s_i = (s_k s_i s_k) s_\varepsilon (s_k s_i s_k) s_k = t_i t_\varepsilon t_i t_k.
\end{align*}

$(g')$ Note that $s_ks_i=s_is_k$ and $s_\varepsilon s_i s_j s_i=s_i s_j s_is_\varepsilon$.We have
\begin{align*}
t_\varepsilon t_i t_j t_k  t_j t_i & = (s_k s_\varepsilon s_k) s_i s_j s_k s_j s_i = s_k s_\varepsilon s_i s_k s_j s_k s_j s_i = s_k (s_\varepsilon s_i s_j s_i)s_k \\
& = s_k (s_i s_j s_is_\varepsilon)s_k = s_i s_j s_k s_j s_i (s_k s_\varepsilon s_k)=t_i t_j t_k t_j t_i t_\varepsilon.
\end{align*}

$(h')$ Note that $s_j s_\varepsilon = s_\varepsilon s_j$ and $ s_\varepsilon s_i s_j s_k s_j s_i = s_i s_j s_k s_j s_i s_\varepsilon$. We have
\begin{align*}
t_\varepsilon t_k t_jt_k &= s_\varepsilon s_k (s_k s_j s_k) s_k = s_\varepsilon s_j = s_j s_\varepsilon= s_k (s_k s_j s_k) s_k s_\varepsilon= t_k t_jt_kt_\varepsilon, \\
t_\varepsilon t_i t_j t_i &= s_\varepsilon s_i (s_k s_j s_k) s_i = s_\varepsilon s_i s_j s_k s_j s_i = s_i s_j s_k s_j s_i s_\varepsilon=s_i (s_k s_j s_k) s_is_\varepsilon=t_i t_j t_it_\varepsilon.
%t_k t_j t_\varepsilon t_j & = s_k (s_k s_j s_k) s_\varepsilon (s_k s_j s_k) = s_j (s_k s_\varepsilon s_k) s_j s_k = s_j (s_i s_\varepsilon s_i) s_j s_k = s_js_i(s_\varepsilon s_i s_j s_k s_js_i)s_is_j\\
%&=s_js_i(s_i s_j s_k s_js_is_\varepsilon)s_is_j= s_k s_j (s_i s_\varepsilon s_i) s_j = (s_k s_j s_k) s_\varepsilon (s_k s_j s_k) s_k = t_j t_\varepsilon t_j t_k, \\
%t_i t_j t_\varepsilon t_j & = s_i (s_k s_j s_k) s_\varepsilon (s_k s_j s_k) = s_i s_k s_j s_i s_\varepsilon s_i s_j s_k = s_k s_i s_j s_i s_\varepsilon s_i s_j s_k \\
%&=s_k s_j s_i s_\varepsilon s_j s_i s_j s_k = s_k s_j s_i s_\varepsilon s_i s_j s_k s_i = (s_k s_j s_k) s_\varepsilon (s_k s_j s_k) s_i = t_j t_\varepsilon t_j t_i.
\end{align*}

$(i')$ Note that $s_k s_i = s_i s_k$ and $s_\varepsilon s_j s_i s_j  = s_j s_i s_j s_\varepsilon$. We have
\begin{align*}
t_\varepsilon t_k t_j t_i t_j t_k & = s_\varepsilon s_k (s_k s_j s_k) s_i (s_k s_j s_k) s_k  = s_\varepsilon s_j s_k s_i s_k s_j  = s_\varepsilon s_j s_i s_j  = s_j s_i s_j s_\varepsilon \\
& = s_k (s_k s_j s_k) s_i (s_k s_j s_k) s_k s_\varepsilon = t_k t_j t_i t_j t_k t_\varepsilon.
\end{align*}

$(j')$ Note that $s_j s_\varepsilon = s_\varepsilon s_j$ and $s_\varepsilon s_k s_j s_i s_j s_k = s_k s_j s_i s_j s_k s_\varepsilon$. We have
\begin{align*}
t_\varepsilon t_j t_k t_j & = (s_k s_\varepsilon s_k) s_j s_k s_j = s_k s_\varepsilon s_j s_k = s_k s_j s_\varepsilon s_k =  s_j s_k s_j (s_k s_\varepsilon s_k) = t_j t_k t_j t_\varepsilon, \\
t_\varepsilon t_j t_i t_j & = (s_k s_\varepsilon s_k) s_j s_i s_j = s_j s_i s_j s_k s_\varepsilon s_k = t_j t_i t_j t_\varepsilon.
\end{align*}

$(k')$ Note that $s_\varepsilon s_j s_i s_j = s_j s_i s_j s_\varepsilon$. We have
$t_\varepsilon t_j t_k t_i t_k t_j= s_\varepsilon s_j s_k (s_k s_i s_k) s_k s_j = s_\varepsilon s_j s_i s_j = s_j s_i s_j s_\varepsilon = s_j s_k (s_k s_i s_k) s_k s_j s_\varepsilon = t_j t_k t_i t_k t_j t_\varepsilon.$

$(l')$ Note that $s_i s_j = s_j s_i$, $s_\varepsilon s_k = s_k s_\varepsilon$, and $s_\varepsilon s_j s_k s_i s_k s_j=s_j s_k s_i s_k s_j s_\varepsilon $. We have
\begin{align*}
t_i t_k t_j t_k & = s_i s_k (s_k s_j s_k) s_k = s_i s_j = s_j s_i = s_k (s_k s_j s_k) s_k s_i = t_k t_j t_k t_i, \\
t_\varepsilon t_j t_i t_j & = s_\varepsilon (s_k s_j s_k) s_i (s_k s_j s_k) = s_k (s_\varepsilon s_j s_k s_i s_k s_j) s_k =  s_k (s_j s_k s_i s_k s_j s_\varepsilon) s_k \\
& = (s_k s_j s_k) s_i (s_k s_j  s_k) s_\varepsilon = t_j t_i t_j t_\varepsilon.
\end{align*}

In order to prove the relation (R4), we need the following lemma.

\begin{lemma}\label{lemma of mutations sequence}
Let $Q \in \mathcal{U}(A^{\varepsilon}_\ell)$ for some $\ell\geq 1$ and $c$ be a fixed vertex in $Q$. Let $k$ be a mutable vertex. Suppose that $(c,\ldots,\varepsilon)$ (respectively, $(c'=c,\ldots,\varepsilon'=\varepsilon)$) is the shortest path from $c$  (respectively, $c'$) to $\varepsilon$  (respectively, $\varepsilon'$) in $Q$  (respectively, $\mu_k(Q)$). Then
\[
P(t_{c'},t_{\varepsilon'})=P(s_c,s_\varepsilon) \text{ or } P(s_c,s_\varepsilon)s_k \text{ or }s_kP(s_c,s_\varepsilon) \text{ or } s_kP(s_c,s_\varepsilon)s_k.
\]
and 
\[
P(t_{\varepsilon'},t_{c'})=P(s_\varepsilon,s_c) \text{ or } s_kP(s_\varepsilon,s_c) \text{ or }P(s_\varepsilon,s_c)s_k \text{ or } s_kP(s_\varepsilon,s_c)s_k.
\]
In particular, $P(t_{c'},t_{\varepsilon'})P(t_{\varepsilon'},t_{c'})=P(s_c,s_\varepsilon)P(s_\varepsilon,s_c)$ or $s_kP(s_c,s_\varepsilon)P(s_\varepsilon,s_c)s_k$.
\end{lemma}
\begin{proof}
{\bf Case 1.} The vertex $k$ does not connect to any vertex in $\{c,\ldots,\varepsilon\}$. Then the shortest path relations keep unchanged. 

{\bf Case 2.} The vertex $k$ connects to a vertex $i\in\{c,\ldots,\varepsilon\}$. Either $P(t_{c'},t_{\varepsilon'})=P(s_c,s_\varepsilon)$ or by the commutative property of $s_k$ and $s_j$, where $j$ goes over all vertices which are not incident to $k$, $P(t_{c'},t_{\varepsilon'})=s_kP(s_c,s_\varepsilon)s_k$. 

{\bf Case 3.} The vertex $k$ connects to two vertices $i,j\in\{c,\ldots,\varepsilon\}$. By Proposition \ref{FZ oriented cycle} and the definition of $\mathcal{\phi}(A^{\varepsilon}_{k})$ in Section \ref{mutations classes with frozen vertices}, either vertices $i$ and $j$ are adjacent or $k\in\{c,\ldots,\varepsilon\}$ and vertices $i$ and $j$ are neighbours of $k$. In the first case, $P(t_{c'},t_{\varepsilon'})=P(s_c,s_\varepsilon)s_k$ or $s_kP(s_c,s_\varepsilon)$. In the later case, either the shortest path relation keeps unchanged or $P(t_{c'},t_{\varepsilon'})=P(s_c,s_\varepsilon)s_k$ or $s_kP(s_c,s_\varepsilon)$ or $s_kP(s_c,s_\varepsilon)s_k$.

{\bf Case 4.} By the definition of $\mathcal{\phi}(A^{\varepsilon}_\ell)$ in Section \ref{mutations classes with frozen vertices}, the vertex $k$ is impossible to connect to three or more vertices in $\{c,\ldots,\varepsilon\}$.

We reverse the order of $P(t_{c'},t_{\varepsilon'})$, we get $P(t_{\varepsilon'},t_{c'})$. 
\end{proof}

In type $B^{\varepsilon}_n$, without loss of generality, we assume that $(0,1,\ldots,d,\varepsilon)$ is the shortest path from $0$ to $\varepsilon$ in $Q$, where the weight between vertex $0$ and vertex $1$ is 2. Without loss of generality, we assume that $(0,1',\ldots,d',\varepsilon)$ is the shortest path from $0$ to $\varepsilon$ in $\mu_k(Q)$, where the weight between vertex $0$ and vertex $1'$ is 2.

In Figure \ref{mutation class of type B}, if $k=0$ or $k=1$ or $k,0,1,$ is an oriented 3-cycle, it is easy to check that $P(t_0,t_\varepsilon)=P(t_{1'},t_\varepsilon)$ and $P(t_\varepsilon,t_0)=P(t_\varepsilon,t_{1'})$. Otherwise, by Lemma \ref{lemma of mutations sequence} and the commutative property of $s_k$ and $s_0$, we have $P(t_0,t_\varepsilon)=P(t_{1'},t_\varepsilon)$ and $P(t_\varepsilon,t_0)=P(t_\varepsilon,t_{1'})$.
 
%by Lemma \ref{lemma of mutations sequence}, we only need to check the following cases: 
In type $D^{\varepsilon}_n$, Lemma \ref{lemma of mutations sequence} allows us to reduce 
the proof of (R4) to the proof of the following cases: For any quiver in Figure \ref{Type I}, we check that $k=a,b,c$, for any quiver in Figures \ref{Type II}--\ref{Type III}, we check that $k=a,b,c,d$, and for any quiver in Figure \ref{Type IV}, we check $k=i_1,i_2,\ldots, i_d, c',c'',c''',\ldots$. 

In Figures \ref{Type I}--\ref{Type III}, if $k=a$, then either the shortest path relations keep unchanged or $P(t_a,t_\varepsilon)P(t_\varepsilon,t_a)=s_aP(s_a,s_\varepsilon)P(s_\varepsilon,s_a)s_a$, $P(t_b,t_\varepsilon)P(t_\varepsilon,t_b)=s_aP(s_b,s_\varepsilon)P(s_\varepsilon,s_b)s_a$ by the commutative relation $s_as_b=s_bs_a$, so $P(t_a,t_\varepsilon)P(t_\varepsilon,t_a)=P(t_b,t_\varepsilon)P(t_\varepsilon,t_b)$. We prove that $k=b$ by interchanging $a$ and $b$. If $k=c$ or $k=d$, then we prove the following case, other possibilities are similar:
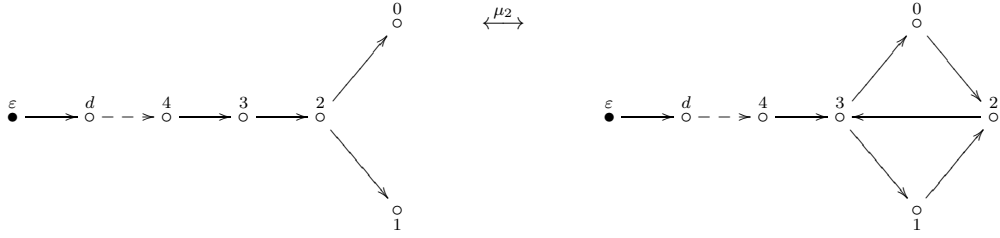
\begin{align*}
\scalemath {0.8} {\xymatrix{&  &  &  & & \overset{0} \circ & \\
\overset{\varepsilon}{\bullet} \ar@{->}[r]  & \overset{d} \circ \ar@{-->}[r] & \overset{4} \circ \ar@{->}[r] & \overset{3} \circ \ar@{->}[r] &  \overset{2} \circ \ar@{->}[ur]\ar@{->}[dr] & \\
& & & & &  \underset{1} \circ & }
\overset{\mu_{2}}{\longleftrightarrow}
\xymatrix{& &  &  &  & \overset{0} \circ \ar@{->}[dr]&  \\
& \overset{\varepsilon}{\bullet} \ar@{->}[r]  & \overset{d}\circ \ar@{-->}[r] & \overset{4} \circ \ar@{->}[r] & \overset{3} \circ  \ar@{->}[dr] \ar@{->}[ur] &  & \overset{2} \circ \ar@{->}[ll] \\
& & & & & \underset{1} \circ \ar@{->}[ur] & }}
\end{align*}

(i) We have
\begin{align*}
P(t_0, t_\varepsilon) P(t_\varepsilon,t_0) & = t_0 t_3 P(t_4, t_\varepsilon) P(t_\varepsilon,t_4) t_3 t_0 = s_0 (s_2s_3s_2) P(s_4, s_\varepsilon) P(s_\varepsilon,s_4)(s_2s_3s_2)s_0 \\
& = P(s_0, s_\varepsilon) P(s_\varepsilon,s_0),\\
P(t_1, t_\varepsilon) P(t_\varepsilon,t_1) & = t_1 t_3 P(t_4, t_\varepsilon) P(t_\varepsilon,t_4) t_3 t_1 = s_1 (s_2s_3s_2) P(s_4, s_\varepsilon) P(s_\varepsilon,s_4)(s_2s_3s_2) s_1 \\
& = P(s_1, s_\varepsilon) P(s_\varepsilon,s_1).
\end{align*}

(ii) We have
\begin{align*}
P(t_0, t_\varepsilon) P(t_\varepsilon, t_0) & = t_0 t_2 P(t_3, t_\varepsilon) P(t_\varepsilon, t_3) t_2  t_0 = (s_2s_0 s_2)s_2 P(s_3, s_\varepsilon) P(s_\varepsilon, s_3) s_2 (s_2s_0 s_2) \\
& = s_2 P(s_0, s_\varepsilon) P(s_\varepsilon,s_0) s_2,\\
P(t_1, t_\varepsilon) P(t_\varepsilon,t_1)& =t_1 t_2 P(t_3, t_\varepsilon) P(t_\varepsilon,t_3)t_2t_1 =(s_2s_1 s_2)s_2 P(s_3, s_\varepsilon) P(s_\varepsilon, s_3) s_2 (s_2s_1 s_2) \\
& = s_2 P(s_1, s_\varepsilon) P( s_\varepsilon,s_1) s_2.
\end{align*}
Finally, in Figure \ref{Type IV}, we prove the following two cases, other possibilities are similar:
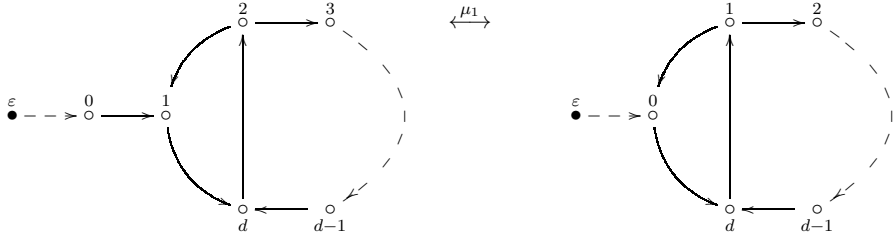
\begin{align*}
\scalemath {0.8} {\xymatrix{ &  & & \overset{2}{\circ}  \ar@{->}[r] \ar@/_10pt/@{->}[dl] \ar@{<-}[dd] & \overset{3}{\circ} \ar@/^35pt/@{-->} [dd]   \\
\overset{\varepsilon}{\bullet} \ar@{-->}[r] & \overset{0}{\circ} \ar[r]   & \overset{1}{\circ} \ar@/_10pt/@{->}[dr] &   & &   \\
& & &  \underset{d}{\circ} \ar@{<-}[r] & \underset{d-1}{\circ}}  \quad \overset{\mu_1}{\longleftrightarrow}
\xymatrix{ &  & & \overset{1}{\circ}  \ar@{->}[r] \ar@/_10pt/@{->}[dl] \ar@{<-}[dd] & \overset{2}{\circ} \ar@/^35pt/@{-->} [dd]   \\
& \overset{\varepsilon}{\bullet} \ar@{-->}[r]   & \overset{0}{\circ} \ar@/_10pt/@{->}[dr] &   & &   \\
& & &  \underset{d}{\circ} \ar@{<-}[r] & \underset{d-1}{\circ} }} \\
\end{align*}

(i) We have $t_3 \cdots t_d P(t_0,t_\varepsilon) P(t_\varepsilon,t_0)  t_d \cdots t_3=s_3 \cdots s_d s_1 P(s_0,s_\varepsilon) P(s_\varepsilon,s_0) s_1 s_d \cdots s_3$, and 
$t_2 t_1 P(t_0,t_\varepsilon) P(t_\varepsilon,t_0) t_1 t_2 =s_1 s_2 s_1 P(s_0,s_\varepsilon) P(s_\varepsilon,s_0) s_1 s_2 s_1 = s_2 s_1 P(s_0,s_\varepsilon) P(s_\varepsilon,s_0)$ $s_1 s_2$.

(ii) We have $t_d t_1 P(t_0,t_\varepsilon) P(t_\varepsilon,t_0) t_1 t_d = s_1 s_d P(s_0,s_\varepsilon) P(s_\varepsilon,s_0) s_d s_1$ and 
\begin{align*}
& t_{d-1} \cdots t_1 P(t_0,t_\varepsilon) P(t_\varepsilon,t_0) t_1 \cdots t_{d-1}= s_{d-1} \cdots s_1 P(s_0,s_\varepsilon) P(s_\varepsilon,s_0) s_1 \cdots s_{d-1} \\
& =s_{d-1} \cdots s_2 s_1 s_2 P(s_0,s_\varepsilon) P(s_\varepsilon,s_0) s_2 s_1 s_2\cdots s_{d-1} \\
&= s_1 (s_{d-1} \cdots s_2 s_1 P(s_0,s_\varepsilon) P(s_\varepsilon,s_0) s_1 s_2 \cdots s_{d-1}) s_1.
\end{align*}

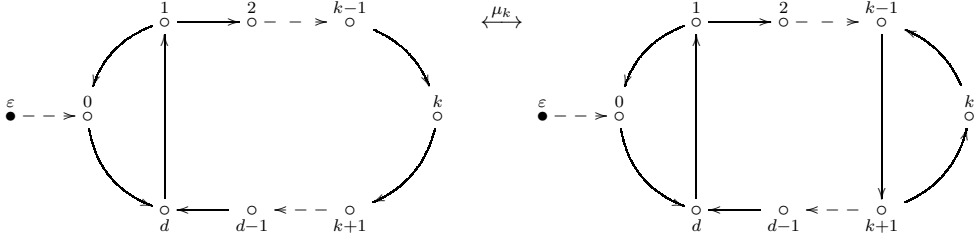
\begin{align*}
\scalemath {0.8} {\xymatrix{&  & & \overset{1}{\circ}  \ar@{->}[r] \ar@/_10pt/@{->}[dl] \ar@{<-}[dd] & \overset{2}{\circ}  \ar@{-->}[r]  & \overset{k-1}{\circ}  \\
& \overset{\varepsilon}{\bullet} \ar@{-->}[r]   & \overset{0}{\circ} \ar@/_10pt/@{->}[dr] &   & & & \overset{k}{\circ} \ar@/_10pt/@{<-}[ul] \\
& & &  \underset{d}{\circ} \ar@{<-}[r] & \underset{d-1}{\circ} & \underset{k+1}{\circ}\ar@/_10pt/@{<-}[ur]  \ar@{-->}[l]}  \quad \overset{\mu_k}{\longleftrightarrow}
\xymatrix{& & \overset{1}{\circ}  \ar@{->}[r] \ar@/_10pt/@{->}[dl] \ar@{<-}[dd] & \overset{2}{\circ} \ar@{-->}[r]  & \overset{k-1}{\circ}  \ar@{->}[dd] \\
\overset{\varepsilon}{\bullet} \ar@{-->}[r]   & \overset{0}{\circ} \ar@/_10pt/@{->}[dr] &   & & & \overset{k}{\circ} \ar@/_10pt/@{->}[ul] \\
& & \underset{d}{\circ} \ar@{<-}[r] & \underset{d-1}{\circ} & \underset{k+1}{\circ}\ar@/_10pt/@{->}[ur]  \ar@{-->}[l] }}
\end{align*}

(i) We have $t_1 P(t_0,t_\varepsilon) P(t_\varepsilon,t_0) t_1 =s_1 P(s_0,s_\varepsilon) P(s_\varepsilon,s_0) s_1$ and 
\begin{gather}
\begin{align*}
& t_2 \cdots t_{k-1} t_{k+1} \cdots t_d P(t_0,t_\varepsilon) P(t_\varepsilon,t_0)  t_d \cdots t_{k+1} t_{k-1} \cdots t_2 \\
& =s_2 \cdots (s_k s_{k-1} s_k) s_{k+1} \cdots s_d P(s_0,s_\varepsilon) P(s_\varepsilon,s_0) s_d \cdots s_{k+1} (s_k s_{k-1} s_k) \cdots s_2 \\
& =s_2 \cdots (s_{k-1} s_k s_{k-1}) s_{k+1} \cdots  s_d P(s_0,s_\varepsilon) P(s_\varepsilon,s_0) s_d  \cdots s_{k+1} (s_{k-1} s_k s_{k-1}) \cdots s_2 \\
& =s_2 \cdots s_{k-1} s_k s_{k+1} \cdots s_d P(s_0,s_\varepsilon) P(s_\varepsilon,s_0) s_d \cdots s_{k+1} s_k s_{k-1} \cdots s_2.
\end{align*}
\end{gather}

(ii) We have $t_1 P(t_0,t_\varepsilon) P(t_\varepsilon,t_0) t_1=s_1 P(s_0,s_\varepsilon) P(s_\varepsilon,s_0) s_1$ and 
\begin{gather}
\begin{align*}
&t_2 \cdots t_{k-1} t_k t_{k+1} \cdots t_d P(t_0,t_\varepsilon) P(t_\varepsilon,t_0) t_d \cdots t_{k+1} t_k t_{k-1} \cdots t_2 \\
&=s_2 \cdots s_{k-1} s_k (s_k s_{k+1} s_k) \cdots s_d P(s_0,s_\varepsilon) P(s_\varepsilon,s_0) s_d \cdots (s_k s_{k+1} s_k) s_k s_{k-1} \cdots s_2 \\
&=s_2 \cdots s_{k-1} s_{k+1} s_k \cdots s_d P(s_0,s_\varepsilon) P(s_\varepsilon,s_0) s_d \cdots s_k s_{k+1} s_{k-1} \cdots s_2 \\
&= s_2 \cdots s_{k-1} s_{k+1} \cdots s_d P(s_0,s_\varepsilon) P(s_\varepsilon,s_0) s_d \cdots s_{k+1}s_{k-1} \cdots s_2.
\end{align*}
\end{gather}

Proposition \ref{inverse monoids homomorphism} is proved. 

\section*{Acknowledgements}
The authors would like to express their gratitude to B. Everitt, W. N. Franzsen, R. Schiffler, and C. C. Xi for helpful discussions. The authors are very grateful to the anonymous reviewer for the comments and valuable suggestions, especially, the reviewer pointed out an error in the former version. B. Duan was supported by China Scholarship Council to visit Department of Mathematics at University of Connecticut and he would like to thank R. Schiffler for hospitality during his visit. This work was partially supported by the National Natural Science Foundation of China (no. 11771191, 11371177, 11501267, 11401275). The research of J.-R. Li on this project is supported by the Minerva foundation with funding from the Federal German Ministry for Education and Research, and by the Austrian Science Fund (FWF): M 2633 Meitner Program.

\end{document}